\RequirePackage{fix-cm}

\documentclass[smallcondensed, numbook, envcountsame]{svjour3} 

\smartqed 

\usepackage[utf8x]{inputenc}
\usepackage[italian, english]{babel} 
\usepackage{graphicx}
\usepackage[all,cmtip]{xy}
\usepackage{amssymb,amsfonts,latexsym}
\usepackage{enumerate}
\usepackage{hyperref} \hypersetup{colorlinks=true, linktoc=page, pdfauthor={Giovanni Stagliano'}, pdftitle={On special quadratic birational transformations} }
\usepackage{breakurl}
\usepackage{mathptmx}

\newcommand{\sS}{{\mathbf{S}}}  
\newcommand{\B}{{\mathfrak{B}}}  
\newcommand{\G}{{\mathcal{G}}}   
\newcommand{\I}{{\mathcal{I}}}   
\renewcommand{\L}{{\mathcal{L}}} 
\renewcommand{\O}{{\mathcal{O}}} 
\newcommand{\Q}{{\mathbf{Q}}}   
\newcommand{\CC}{{\mathbb{C}}}   
\newcommand{\GG}{{\mathbb{G}}}   
\newcommand{\ZZ}{{\mathbb{Z}}}   
\newcommand{\PP}{{\mathbb{P}}}   
\newcommand{\FF}{{\mathbb{F}}}
\newcommand{\E}{{\mathcal{E}}}
\newcommand{\Sec}{{\mathrm{Sec}}} 
\newcommand{\Bl}{{\mathrm{Bl}}}   
\newcommand{\reg}{{\mathrm{reg}}}  
\newcommand{\sing}{{\mathrm{sing}}} 
\newcommand{\Pic}{{\mathrm{Pic}}} 
\newcommand{\rk}{{\mathrm{rk}}}  
\newcommand{\Hilb}{{\mathrm{Hilb}}} 
\newcommand{\Graph}{{\mathrm{Graph}}} 

\makeatletter

\newcommand{\Rmnum}[1]{\expandafter\@slowromancap\romannumeral #1@}
\makeatother

\spnewtheorem{case}{Case}[section]{\itshape}{\rmfamily} 
\spnewtheorem{subcase}{Subcase}[case]{\itshape}{\rmfamily}
\spnewtheorem{step}{Claim}[section]{\itshape}{\rmfamily}

\journalname{Noname}

\begin{document}

\title{On special quadratic birational transformations of a projective space into a hypersurface}
\titlerunning{On special quadratic birational transformations}

\author{Giovanni Staglian\`o} 
\authorrunning{G. Staglian\`o} 

\institute{Giovanni Staglian\`o \at
           Ph.D. in Mathematics (\Rmnum{23} cycle) -  University of Study of Catania \\
           \email{\href{mailto:gstagliano@dmi.unict.it}{gstagliano@dmi.unict.it}} 
}

\date{\today \ (The date of receipt and acceptance will be inserted by the editor)} 

\maketitle

\begin{abstract}
We study transformations as in the title 
with emphasis on those having smooth connected base locus, called ``special''.
In particular, 
we classify all special quadratic birational maps into a quadric hypersurface whose inverse 
is given by quadratic forms by showing 
that there are only four examples having general
 hyperplane sections of  Severi varieties as base loci.  
\end{abstract}

\tableofcontents

\section*{Introduction}
\addcontentsline{toc}{section}{Introduction}
A birational transformation
 $\varphi:\PP^n\dashrightarrow\sS\subset\PP^{n+1}$
of a complex projective space
 into a nonlinear and sufficiently regular hypersurface  $\sS$,
is said to be of type $(d_1,d_2)$,
if $\varphi$ (resp. $\varphi^{-1}$) is defined by a sublinear system 
of $|\O_{\PP^n}(d_1)|$  (resp. $|\O_{\sS}(d_2)|$); $\varphi$ is said to be special if 
its base locus
 $\B=\mathrm{Bs}(\varphi)$ is  smooth and connected.
Consider  a  
 $\varphi$ as above, supposed to be special and  of type $(2,d)$.
 The blow-up
 $\pi:X=\Bl_{\B}(\PP^n)\rightarrow\PP^n$ of $\PP^n$ along the base locus 
 $\B$
resolves the indeterminacies of the transformation
 $\varphi$. 
Comparing the two ways in which it is possible to write
 the canonical class
$K_X$, with respect to $\pi$ and $\pi':=\varphi\circ\pi$, 
we get 
a formula expressing  the dimension of $\B$ as a function of $n$, $d$ and $\deg(\sS)$,
see  Proposition \ref{prop: dimension formula}.
The primary advantage in dealing with the case
$d_1=2$ is that $\B$ is 
a $QEL$-variety, i.e. 
the general entry locus  of $\B$ is a smooth quadric hypersurface. 
Therefore it is possible to apply the main results  of the theory of  
 $QEL$-varieties; in particular,
 the Divisibility Theorem \cite[Theorem~2.8]{russo-qel1}, 
together with the
formula on the dimension 
of $\B$, 
drastically reduces the set of quadruples 
 $(\dim(\B),n,d,\deg(\sS))$ for which  such a
 $\varphi$ exists.

The $\varphi$ of type $(2,1)$   
are  described in  Proposition \ref{prop: type 2-1} as
the only 
 transformations whose base locus is a quadric of codimension
 $2$; 
in particular, 
modulo projective transformations, 
there is only one example.
With regard to the
 $\varphi$'s of type $(2,2)$ into a quadric hypersurface $\Q$,
it is known that a special Cremona 
 transformation 
$\PP^{n+1}\dashrightarrow\PP^{n+1}$ of type $(2,2)$ 
has as base locus a Severi variety (see \cite{ein-shepherdbarron}) and moreover,
modulo projective transformations, 
there exist only four Severi varieties
 (see \cite{zak-tangent} and Table \ref{tab: severi varieties} below). 
Now, if we restrict these transformations 
to a general hyperplane $\PP^{n}\subset\PP^{n+1}$,
we clearly obtain special transformations  
$\varphi:\PP^n\dashrightarrow\Q\subset\PP^{n+1}$ of type $(2,2)$ 
(see Example \ref{example: d=2 Delta=2}).  
In Theorem \ref{prop: classification of type 2-2 into quadric} 
we  prove 
 that all examples of $\varphi$ of type $(2,2)$ into a quadric $\Q$ arise in this way;
in particular, their  base loci
are hyperplane sections of Severi varieties. 
Regarding special transformations of type $(2,2)$ into a cubic and quartic hypersurface
 we are able  to determine some  invariants of the base locus  as
the Hilbert polynomial and the Hilbert scheme of lines passing through a point
(Propositions \ref{prop: invariants d=2 Delta=3} and \ref{prop: invariants d=2 Delta=4}).

We point out that in \cite[\S 5]{russo-simis} (see also \cite{semple}), 
in a similar way to here,
special birational transformations 
$\PP^{2m-2}\dashrightarrow\PP^{{m+1 \choose 2}-1} $ of type $(2,1)$ and of type $(2,2)$
whose image is the Grassmannian $\GG(1,m)$ were classified.

Another approach to  the study  
of all $\varphi$ of type $(2,d)$ is their classification according to
 the dimension of the base locus. 
In Table
 \ref{tab: cases with dimension leq 3}  
(which is constructed 
via Propositions \ref{prop: 2-fold in P6} and \ref{prop: 3-fold in P8}), 
we provide a list of all possible base loci when the dimension is at most $3$,
although in one case we do not know if it really exists.
As a consequence, in Corollary \ref{prop: classification type 2-3 into cubic}, we obtain
that a special transformation  
 $\varphi$ of type $(2,3)$ into a cubic hypersurface
 $\sS$ has as base locus  
the blow-up $\Bl_{\{p_1,\ldots,p_5\}}(Q)$ of $5$ points $p_1,\ldots,p_5$ 
in a smooth quadric $Q\subset\PP^4$, embedded in   
$\PP^8$ by the linear system $|2 H_{\PP^4}|_{Q}-p_1-\cdots-p_5|$ 
(see also Example \ref{example: d=3 Delta=3}).

In the  last section we treat  birational 
transformations  $\varphi:\PP^n\dashrightarrow\Q\subset\PP^{n+1}$ 
 of type $(2,d)$ into a quadric, 
having  reduced base locus  and with $n\leq 4$. 
For $n=3$ and $d\geq 2$ we prove that $d=2$ and there exists
only one example  
which is not  special  (Proposition \ref{prop: P3});
for $n=4$ and $d\geq 2$ we get $d=2$ and all the examples are degenerations of  the unique special case (Proposition \ref{prop: P4}).

\section{Reminders and notation}\label{sec: preliminaries}
Let $X\subset\PP^n$ be a complex projective irreducible algebraic variety
of dimension $r$ and
embedded in the projective space $\PP^n$.
We shall assume $X\subset\PP^n$ nondegenerate,
i.e. not contained in any hyperplane.
\paragraph{Secant variety}   
Denote by $S_{X}$ the closure of $S^0_{X}$ in $X\times X\times\PP^n$, where
$
S^0_{X}:=\{(x,x',p) \in X\times X\times\PP^n: x\neq x' \mbox{ and } p\in\langle x,x'\rangle \}
$,
and let $\pi_{\PP^n}:S_X\rightarrow\PP^n $, $\pi_{X,X}:S_X\rightarrow X\times X$, $p_{X}:X\times X\rightarrow X$ be the projections.
The \emph{secant variety} of $X$, $\Sec(X)$, is defined
as
the scheme-theoretic  image of $\pi_{\PP^n}$, i.e. as
$$\Sec(X):=\overline{\pi_{\PP^n}(S_{X})}=\overline{\bigcup_{(x,x')\in X\times X \atop  x\neq x'} \langle x,x'\rangle}.$$
Note that $S_{X}$ is an irreducible variety of
dimension $2\dim(X)+1$, hence $\Sec(X)$ is also irreducible
of dimension
$\dim(\Sec(X))\leq 2\dim(X)+1$.
A line $l$ is said to be a \emph{secant} of $X$ if the length of the scheme $X\cap l$ is at least $2$;
in these terms, $\Sec(X)$ is the union of all secant lines of $X$.
The \emph{secant defect} of $X$ is defined as the nonnegative integer
$$
\delta(X):=2\dim(X)+1-\dim(\Sec(X)) 
$$
and $X$ is called \emph{secant defective} if $\delta(X)>0$. 
An essential tool for studying secant defective varieties 
is the so-called Terracini Lemma
(see \cite[Theorem~1.3.1]{russo-specialvarieties}), which states that
for two general points $x, x'\in X$ ($x\neq x'$) and for a general point $p\in\langle x,x'\rangle $,
we have
$$
T_p(\Sec(X))=\langle T_x(X), T_{x'}(X)\rangle.
$$
\paragraph{Contact loci}   For
$p\in\Sec(X)\setminus X$,
put
\begin{eqnarray*}
L_p(X)&:=&\pi_{\PP^n}(\pi_{X,X}^{-1}(\pi_{X,X}(\pi_{\PP^n}^{-1}(p)))) 
        = \bigcup\{l: l \mbox{ secant line of } X \mbox{ through }p\},\\
\Sigma_p(X)&:=& \pi_X(\pi_{\PP^n}^{-1}(p))=L_p(X)\cap X  
             =  \overline{\{x\in X: \exists x'\in X \mbox{ with } x\neq x' \mbox{ and } p\in \langle x,x' \rangle \}},
\end{eqnarray*}
where $\pi_X=p_X\circ\pi_{X,X}$.
$\Sigma_p(X)$ is called the \emph{entry locus}
of $X$ with respect to $p$.
Note that
 $L_p(X)$ is a cone over $\Sigma_p(X)$ of vertex $p$, hence
 $\dim(L_p(X))=\dim(\Sigma_p(X))+1$. 
If $p\in\Sec(X)\setminus X$ is a general point,  then
$\Sigma_p(X)$ is reduced 
and equidimensional of
dimension  $\delta(X)$.
The \emph{tangential contact locus} of
 $X$ with respect to the general point  $p$ 
(see also \cite[Definition~2.3.3]{russo-specialvarieties} 
and \cite[Definition~3.4]{chiantini-ciliberto}) is defined
as
$$
\Gamma_{p}(X):=\overline{\{x\in \reg(X): T_x(X)\subseteq T_p(\Sec(X)) \} }.
$$
By  Terracini Lemma 
it follows  that
$\Sigma_{p}(X)\subseteq \Gamma_p(X)$ and
it is possible to prove that
 for 
 $x,x'\in X$ general points and $p\in\langle x,x' \rangle$ general point, 
the irreducible components of
$\Gamma_p(X)$ through $x$ and $x'$ are uniquely determined and have
 the same
dimension, independent of  $p$. Denoting this dimension with 
 $\gamma(X)$, we have
$\delta(X)\leq\gamma(X)$.
\paragraph{Second fundamental form}
For a general point 
 $x\in X$  consider the projection of  
 $X$ from the tangent space  $T_x(X)\simeq \PP^r$ onto a linear space   
$\PP^{n-r-1}$ skew to $T_x(X)$,
$$
\tau_{x,X}:X\subset\PP^n\dashrightarrow \overline{\tau_{x,X}(X)}=W_{x,X}\subseteq\PP^{n-r-1}.
$$
By Terracini Lemma it follows that
$W_{x,X}\subset\PP^{n-r-1}$ is an irreducible nondegenerate variety 
 of dimension $r-\delta(X)$ (see \cite[Proposition~1.3.8]{russo-specialvarieties}).
Now, consider the blow-up
 of $X$ at the general point  $x$,
$
\pi_x:\Bl_x(X)\rightarrow X
$,
and denote by 
 $E=E_X=E_{x,X}=\PP^{r-1}$ its exceptional divisor and by 
 $H$ a divisor of the linear system $|\pi_{x}^{\ast}(\O_X(1))|$.
Since $X$  is not linear, it is defined the rational map
$$
\phi: E{\dashrightarrow} W_{x,X}\subset\PP^{n-r-1},
$$
and the linear system associated to 
 $\phi$, i.e.
 $|H-2E|_{|E}\subseteq |-2E|_{|E}=|\O_{\PP^{r-1}}(2)|$, is called 
\emph{second fundamental form}  of $X$ (at the point $x$) and it is denoted by
 $|II_{x,X}|$.
Of course, $\dim(|II_{x,X}|)\leq n-r-1$ and 
the image of $\phi=\phi_{|II_{x,X}|}$ is contained in $W_{x,X}$.
If
 $X$ is smooth,
secant defective  
and $\Sec(X)\subsetneq\PP^n$, then by \cite[Proposition~2.3.2]{russo-specialvarieties}
$\overline{\phi_{|II_{x,X}|}(E)}=W_{x,X}$ and
 in particular $\dim(|II_{x,X}|)=n-r-1$.
If $x\in X$, we denote by $\L_{x,X}$ the Hilbert scheme of lines contained 
in $X$ and passing through the point 
$x$. If $x\in\reg(X)$, there exists a natural 
closed embedding of $\L_{x,X}$ into the exceptional divisor 
 $E_{x,X}$ (see \cite{russo-linesonvarieties}).
Moreover, 
if $B_{x,X}=\mathrm{Bs}(|II_{x,X}|)$ is the base locus of 
   the second fundamental form at a point $x\in\reg(X)$, then there exists 
a natural closed embedding of
 $\L_{x,X}$ into $B_{x,X}$, which is an isomorphism if 
    $X$ is defined by 
quadratic forms (see \cite[Corollary~1.6]{russo-linesonvarieties}).
\paragraph{Gauss map} Recall that the Gauss map of $X$
 is the rational map 
$\G_{X} :X\dashrightarrow \GG(r,n)$, which sends the point 
 $x\in\reg(X)$ to the tangent space $T_x(X)\in \GG(r,n)$.
For a general point 
 $x\in\reg(X)$, the closure of the fiber 
$\overline{\G_X^{-1}(\G_X(x))}=\overline{\{y\in \reg(X): T_y(X)=T_x(X)\}}$
is a linear space and when 
 $X$ is smooth, 
$\G_X:X\dashrightarrow \G_X(X)$ is birational (see \cite{zak-tangent}).
\paragraph{$(L)QEL$-varieties}
The variety $X\subset\PP^n$ is called \emph{$QEL$-variety} (\emph{quadratic entry locus variety}) of type $\delta=\delta(X)$ if 
the general entry locus $\Sigma_p(X)$ is a quadric hypersurface (of dimension $\delta$);
$X$ is called \emph{$LQEL$-variety} (\emph{locally quadratic entry locus variety}) if every 
irreducible component of the general entry locus 
  $\Sigma_p(X)$ is a quadric hypersurface; 
$X$ is called \emph{$CC$-variety} (\emph{conic-connected variety})
if for two general points  $x,y\in X$ there exists an irreducible
conic  $C_{x,y}$ contained in $X$ and passing through $x,y$
(see also \cite{russo-qel1}, \cite{ionescu-russo-conicconnected}, \cite{zak-tangent}).
Of course a $QEL$-variety is a $LQEL$-variety and a $LQEL$-variety of type $\delta>0$ is
a $CC$-variety, but the converse implications are not true in general. 

For $X\subset\PP^n$, with $\Sec(X)\subsetneq\PP^n$ and
$\delta(X)>0$, define the nonnegative integer
 $\widetilde{\gamma}(X)$ as the dimension of the general fiber
of the Gauss map
$G_{W_{x,X}}:W_{x,X}\dashrightarrow\GG(r-\delta(X), n-r-1)$, 
where $x\in X$ is a general point;
by \cite[Lemma~2.3.4]{russo-specialvarieties} we get 
$\widetilde{\gamma}(X)=\gamma(X)-\delta(X)$.
A sufficient condition for $X$ to be a $LQEL$-variety is that
 $\widetilde{\gamma}(X)=0$, 
this follows from the so-called Scorza Lemma \cite[Theorem~2.3.5]{russo-specialvarieties}.

Smooth $LQEL$-varieties  of type $\delta\geq3$ 
are 
 subject to a strong numerical constraint  
given by the Divisibility Theorem \cite[Theorem~2.8]{russo-qel1},
which states that,  under these hypotheses, we have
$$
\dim(X)\equiv \delta(X)\ \mathrm{mod}\ 2^{\lfloor (\delta(X)-1)/2 \rfloor}.
$$ 
\paragraph{Severi varieties}
If the variety
 $X\subset\PP^n$ is smooth and $\Sec(X)\subsetneq\PP^n$,
by \cite[\Rmnum{2} Corollary~2.11]{zak-tangent} it follows
$n \geq (3r+4)/2$ and
$X$ is called \emph{Severi variety}  if $n = (3r+4)/2$.
Severi varieties are classified, 
modulo projective transformations,
in  Table
 \ref{tab: severi varieties} 
(for the proof see \cite{lazarsfeld-vandeven} and \cite{zak-tangent}).
\begin{table}[htbp]
\begin{center}
\begin{tabular}{|c|c|c|c|}
\hline
$\dim(X)$& $h^0(\I_{X}(2))$ & $X$ & $\L_{x,X}$\\
\hline
 \hline
2& 6 & Veronese surface   $\nu_2(\PP^2)\subset\PP^5$ & $\emptyset$\\
 \hline
4& 9 & Segre variety  $\PP^2\times\PP^2\subset\PP^8$ & $\PP^1\sqcup\PP^1\subset \PP^3$\\
 \hline
8& 15 & Grassmannian $\GG(1,5)\subset\PP^{14}$ & $\PP^1\times\PP^3\subset\PP^7$\\
 \hline
16& 27 & Cartan variety $E_6\subset\PP^{26}$ & $S^{10}\subset\PP^{15}$\\
\hline
\end{tabular}
\end{center}
\caption{Severi varieties.}
\label{tab: severi varieties}
\end{table}
Severi varieties 
are also characterized in 
 \cite{ein-shepherdbarron} as the only
smooth and connected 
base loci of 
quadro-quadric Cremona transformations 
(i.e. birational transformations of $\PP^n$ defined by quadratic forms
and having inverse of the same kind).

\section{Transformations of type \texorpdfstring{$(2,1)$}{(2,1)}}\label{sec: type 2-1}
\begin{definition}
A birational transformation 
 $\varphi:\PP^n\dashrightarrow\overline{\varphi(\PP^n)} =\sS\subset\PP^{n+1}$ 
is said to be of type  $(2,d)$ if it is quadratic 
(i.e. defined by a linear system of quadrics without fixed component) 
and $\varphi^{-1}$ can be defined by a linear system contained in $|\O_{\sS}(d)|$,
with $d$ minimal with this property;
$\varphi$ is said to be \emph{special} if its base locus  is smooth and connected.
\end{definition}
\begin{lemma}\label{prop: cohomology I2B} 
Let $\varphi:\PP^n\dashrightarrow\sS\subset\PP^{n+1}$ be 
a quadratic birational transformation, 
with base locus $\B$ and image a normal nonlinear hypersurface $\sS$. 
Then  $h^0(\PP^{n}, \I_{\B}(2))=n+2$. 
\end{lemma}
\begin{proof} 
 Resolve the indeterminacies of $\varphi$ with the diagram 
\begin{equation}\label{eq: diagram resolving map} 
\xymatrix{ & X \ar[dl]_{\pi} \ar[dr]^{\pi'}\\ \PP^n\ar@{-->}[rr]^{\varphi}& & \sS } 
\end{equation}
where $\pi:X=\Bl_{\B}(\PP^n)\rightarrow\PP^n$ is the blow-up of
 $\PP^n$ along $\B$,   $E$ the exceptional divisor,
 $\pi'=\varphi\circ\pi$.     
By  Zariski's Main Theorem (\cite[\Rmnum{3} Corollary~11.4]{hartshorne-ag} or \cite[\Rmnum{3} \S  9]{mumford})
we have ${\pi'}_{\ast}(\O_{X})=\O_{\sS}$ and by  projection formula 
 \cite[\Rmnum{2} Exercise~5.1]{hartshorne-ag} it follows
${\pi'}_{\ast}({\pi'}^{\ast}(\O_{\sS}(1)))=\O_{\sS}(1)$.
Now, putting $V\subseteq H^0(\PP^n,\O_{\PP^n}(2))$ 
the linear vector space associated to the linear system $\sigma$ defining $\varphi$,
 we have the natural inclusions
\begin{eqnarray*}
 V & \hookrightarrow & H^0(\PP^n,\I_{\B}(2)) 
 \hookrightarrow   H^0(X,\pi^{\ast}(\O_{\PP^n}(2))\otimes\pi^{-1}\I_{\B}\cdot \O_{X}) \\
&\stackrel{\simeq}{\rightarrow} & H^0(X,\pi^{\ast}(\O_{\PP^n}(2))\otimes\O_{X}(-E))
 \stackrel{\simeq}{\rightarrow}  H^0(X,{\pi'}^{\ast}(\O_{\sS}(1))) \\
 & \stackrel{\simeq}{\rightarrow} & H^0(\sS,\O_{\sS}(1)) \stackrel{\simeq}{\rightarrow}  H^0(\PP^{n+1},\O_{\PP^{n+1}}(1)).
\end{eqnarray*}
All these inclusions are isomorphisms, since $\dim(V)=n+2=h^0(\PP^{n+1},\O_{\PP^{n+1}}(1))$.
\end{proof}

\begin{proposition}\label{prop: type 2-1} 
Let $\varphi,\B,\sS$ be as in Lemma \ref{prop: cohomology I2B}.
The following conditions are equivalent:
\begin{enumerate}
 \item\label{item: 1, 2-1} $h^0(\PP^n,\I_{\B}(1))\neq0$; 
 \item\label{item: 2, 2-1} $\B$ is a  quadric of codimension $2$;
 \item\label{item: 3, 2-1} $\B$ is a complete intersection;
 \item\label{item: 4, 2-1} $\varphi$ is of type $(2,1)$.
\end{enumerate}
If one of the previous conditions is satisfied, then 
$\sS$ is a quadric and
 $\rk(\B)=\rk(\sS)-2$.     
\end{proposition}
\begin{proof} 
$(\ref{item: 1, 2-1}\Rightarrow\ref{item: 2, 2-1}\mbox{ and }\ref{item: 3, 2-1})$.
 If $f\in H^0(\PP^n,\I_{\B}(1))$ is a nonzero linear form, 
then the quadrics
 $x_0f,\ldots,x_n f$  
generate a subspace of codimension $1$
 of $H^0(\PP^n,\I_{\B}(2))$,   by Lemma \ref{prop: cohomology I2B}. Hence there exists a 
quadric $F$ such that
$H^0(\PP^n,\I_{\B}(2))=\langle F, x_0f,\ldots,x_nf\rangle$
and $\B=V(F,f)\subset V(f)\subset\PP^{n}$.

\noindent
$(\ref{item: 2, 2-1}\Rightarrow\ref{item: 1, 2-1}\mbox{ and }\ref{item: 4, 2-1})$. 
$\B$ is necessarily degenerate and,
modulo a change of coordinates
on $\PP^{n}$, 
we can suppose
 $\B=V(x_0^2+\cdots+x_s^2, x_n)$,   with $s\leq n-1$.
Hence 
$$
\varphi([x_0,\ldots,x_n])=
[x_0x_n,\ldots,x_n^2,x_0^2+\cdots+x_s^2]=[y_0,\ldots,y_{n+1}],
$$
$$
\sS=V(y_0^2+\cdots+y_s^2-y_{n}y_{n+1}),\quad \varphi^{-1}([y_0,\ldots,y_{n+1}])=[y_0,\ldots,y_{n}]. 
$$

\noindent
$(\ref{item: 4, 2-1}\Rightarrow\ref{item: 2, 2-1})$.     We can suppose that
 $\varphi^{-1}$ is the projection from the point
 $[0,\ldots,0,1]$.     Thus, 
if $\varphi=(F_0,\ldots,F_{n+1})$,   then $F_0,\ldots,F_n$ have
a common factor
 $f$ and  $\B=V(F_{n+1},f)$.

\noindent
$(\ref{item: 3, 2-1}\Rightarrow\ref{item: 1, 2-1})$.  If $h^0(\PP^n,\I_{\B}(1))=0$ and $\B$ is 
a complete intersection,  then every minimal system of generators of the ideal of
 $\B$ consists of forms of degree $2$, but then  
$h^0(\PP^n,\I_{\B}(2))=n-\dim(\B)<n+2$,   
absurd.    
\end{proof}

From now on, unless otherwise specified and except in \S \ref{sec: nonspecial case},
we keep the following notation:
\begin{notation}\label{notation: factorial hypersurface}
 Let $n\geq3$ and 
$\varphi:\PP^{n}\dashrightarrow\overline{\varphi(\PP^n)}=\sS\subset\PP^{n+1}$ be
a special birational transformation  of type $(2,d)$, with $d\geq 2$ and 
with $\sS$ a factorial hypersurface of degree $\Delta\geq 2$
(in particular, when $\Delta=2$, it is enough to require 
 $\rk(\sS)\geq 5$).  
Observe that
 $\Pic(\sS)=\ZZ\langle\O_{\sS}(1)\rangle$ (see \cite[\Rmnum{4} Corollary~3.2]{hartshorne-ample})
and $\omega_{\reg(\sS)}\simeq \O_{\reg(\sS)}(\Delta-n-2)$.
Moreover, denote by  $\B\subset\PP^n$ and $\B'\subset\sS\subset\PP^{n+1}$ 
respectively the base locus  of $\varphi$ and $\varphi^{-1}$
and assume 
\begin{equation}\label{eq: hypothesis on sing locus hypersurface}
(\B')_{\mathrm{red}}\neq (\sing(\sS))_{\mathrm{red}}.
\end{equation}
Put $r=\dim(\B)$, $r'=\dim(\B')$, $\delta=\delta(\B)$ the secant defect, 
$\lambda=\deg(\B)$, $g=g(\B)$ the sectional genus, $P_{\B}(t)$ the Hilbert polynomial,
$i(\B)$ and $c(\B)$ respectively index and coindex (when $\B$ is a Fano variety).
\end{notation}

\section{Properties of the base locus}\label{sec: properties base locus}
\begin{proposition}\label{prop: dimension formula}  
\begin{enumerate} 
\item $\B$ is a
$QEL$-variety of dimension and  type given by
\begin{displaymath}
r=\frac{d\,n-\Delta-3\,d+3}{2\,d-1},\quad
\delta =\frac{n-2\,\Delta-2\,d+4}{2\,d-1}.
\end{displaymath}
 \item $\B'$ is irreducible, generically reduced, of dimension
\begin{displaymath}
r'=\frac{2\left( d\,n-n+\Delta-d-1\right) }{2\,d-1}.
\end{displaymath}
\end{enumerate}
\end{proposition}
\begin{proof} See also \cite[Proposition~2.1]{ein-shepherdbarron}.
Consider the  diagram (\ref{eq: diagram resolving map}), 
where $\pi:X=\Bl_{\B}(\PP^n)\rightarrow\PP^n$
  and $\pi'=\varphi\circ\pi$.     
$X$   
can be identified
 with the graphic 
 $\Graph(\varphi)\subset\PP^n\times\sS$ and the
maps $\pi$,   $\pi'$ 
can be identified
 with the projections onto the factors.
It follows that $(\B)_{\mathrm{red}}$ (resp. $(\B')_{\mathrm{red}}$) is the set 
of the points $x$ such that the fiber $\pi^{-1}(x)$ (resp. ${\pi'}^{-1}(x)$) has 
positive dimension. Denote by $E$ the exceptional divisor of $\pi$, 
  $E'={\pi'}^{-1}(\B')$, 
$H\in|\pi^{\ast}(\O_{\PP^{n}}(1))|$,   
$H'\in|{\pi'}^{\ast}(\O_{\sS}(1))|$ 
and  note that, by the factoriality of 
 $\sS$ and by the proof of
\cite[Proposition~1.3]{ein-shepherdbarron}, it follows
that $E'$ is an irreducible divisor, in particular $\B'$ is irreducible.
Moreover we have the relations
\begin{equation}\label{eq: HH'}
H\sim d\,H'-E',\quad H'\sim 2\,H-E,
\end{equation}
from which we get
\begin{equation}\label{eq: EE'}
E\sim(2\,d-1)\,H'-2\,E',\quad E'\sim (2\,d-1)\,H-d\,E,
\end{equation}
and in particular by (\ref{eq: EE'}) it follows  $E'=(E')_{\mathrm{red}}$.
Put 
\begin{equation}\label{eq: positions}
U=\reg\left(\sS\right)\setminus\sing\left(\left(\B'\right)_{\mathrm{red}}\right),\quad V={\pi'}^{-1}(U),\quad Z=U \cap (\B')_{\mathrm{red}}.
\end{equation}
Observe 
 that, since $X$ is smooth and we have assumed 
  (\ref{eq: hypothesis on sing locus hypersurface}),
we have
$Z\neq\emptyset$. Thus, by \cite[Theorem~1.1]{ein-shepherdbarron}, 
 ${\pi'}|_{V}:V\rightarrow U$ coincides
 with the blow-up  of $U$ along $Z$;
in particular
$$
\Pic(X)=\ZZ\langle H\rangle \oplus\ZZ\langle E\rangle=\ZZ\langle H'\rangle\oplus\ZZ\langle E'\rangle=\Pic(V).
$$
Now, by \cite[\Rmnum{2} Exercise~8.5]{hartshorne-ag} and (\ref{eq: HH'}) and (\ref{eq: EE'}), we have
\begin{eqnarray}\label{eq: K_X}
K_X&\sim&\pi^{\ast}(K_{\PP^n}) + (\mathrm{codim}_{\PP^n}(\B)-1)\,E  \nonumber \\
 & \sim & (-n-1)\,H+(n-r-1)\,E  \nonumber  \\
 & \sim & (-2\,d\,r+r+d\,n-n-3\,d+1)\,H'+(2\,r-n+3)\,E' 
\end{eqnarray}
and also
\begin{eqnarray}\label{eq: K_V}
 K_X|V \sim K_V&\sim&{\pi'}^{\ast}(K_{\reg(\sS)}|U) + (\mathrm{codim}_{U}(Z)-1)\,E' \nonumber\\
   & \sim & (\deg(\sS)-(n+1)-1)\,H' + (\mathrm{codim}_{\sS}(\B')-1)\,E' \nonumber \\
   & \sim & (\Delta-n-2)\,H'+(n-r'-1)\,E'. 
\end{eqnarray}
By the comparison of
 (\ref{eq: K_X}) and (\ref{eq: K_V})
we obtain the expressions of
 $r=\dim(\B)$ and $r'=\dim(\B')$.
Moreover, the proof of \cite[Proposition~2.3(a),(b)]{ein-shepherdbarron} 
adapts to our case,
  producing that the secant  variety  $\Sec(\B)$ is a hypersurface of degree
 $2d-1$ in $\PP^n$ and  $\B$ is a $QEL$-variety of type $\delta=n-r'-2$.
Finally,
 we have  $\B'\cap U=(\B')_{\mathrm{red}}\cap U$ 
by the argument in \cite[\S  2.2]{ein-shepherdbarron}. 
\end{proof}

\begin{remark}\label{remark: dimension formula without hypothesis}
Proceeding as above and
interchanging the rules in (\ref{eq: positions}) by
\begin{equation}
U=\reg\left(\sS\right)\setminus\left(\B'\right)_{\mathrm{red}},\quad V={\pi'}^{-1}(U),\quad Z=U \cap (\B')_{\mathrm{red}}=\emptyset ,
\end{equation}
we deduce that $\pi'|_{V}:V\rightarrow U$ is an isomorphism  and in 
particular $\Pic(V)={\pi'}^{\ast}(\Pic(U))=\ZZ\langle H' \rangle$
and $K_V={\pi'}^{\ast}(K_U)=(\Delta-n-2) H'$.
Now, by (\ref{eq: K_X}),
$K_X|_V = K_{V} \sim (-2\,d\,r+r+d\,n-n-3\,d+1)H'$,
and hence, also without to assume 
 (\ref{eq: hypothesis on sing locus hypersurface}), we get  the formula
$r=\left(dn-\Delta-3d+3\right)/\left(2d-1\right)$.
\end{remark}

Lemma \ref{prop: cohomology twisted ideal} is slightly stronger than
what is obtained by applying
directly the main result in \cite{bertram-ein-lazarsfeld}. However
it is essential to study $\B$  in the case $\delta=0$.  
Note that the assumptions on $\sS$ are not necessary.
\begin{lemma}[\cite{mella-russo-baselocusleq3}]\label{prop: cohomology twisted ideal}
 For $i>0$ and $t\geq n-2r-1$ we have
$
H^i(\PP^n, \I_{\B,\PP^n}(t))=0
$.
\end{lemma}
\begin{proof} 
The proof is located in
\cite[page~6]{mella-russo-baselocusleq3}, 
but we report it for the reader's convenience.
We use the notation of the proof of Proposition 
 \ref{prop: dimension formula}. For $i>0$ and $t\geq n-2r-1$ we have
 \begin{eqnarray*}
  H^i(\PP^n,\I_{\B,\PP^n}(t)) & = & H^i(\PP^n, \O_{\PP^n}(t)\otimes\I_{\B,\PP^n}) \\
 \mbox{(\cite[page~592]{bertram-ein-lazarsfeld})} &= & H^i(X, tH-E) \\
                                 &=& H^i(X, K_X+(tH-E-K_X)) \\
                                 &=& H^i(X, K_X+(t+n+1)H+(r-n)E) \\
                                 &=& H^i(X, K_X+(n-r)(2H-E)+(t-n+2r+1) H) \\
 \mbox{(\cite[page~20]{debarre})} &=& H^i(X, K_X+ \underbrace{
\underbrace{(n-r)H'}_{\mathrm{nef\ and\ big}}+\underbrace{(t-(n-2r-1))H}_{\mathrm{nef}}}_{\mathrm{nef\ and\ big}}) \\
 \mbox{(Kodaira Vanishing Theorem)}                 &=& 0.
 \end{eqnarray*}
\end{proof}
\begin{proposition}\label{prop: cohomology properties}   
\begin{enumerate}
 \item\label{part: B is linearly normal} $\B$ is nondegenerate 
and projectively normal.
 \item\label{part: B is fano} If $\delta>0$, then either $\B$ is a Fano variety 
 of the first species  (that is with  $\Pic(\B)\simeq \ZZ\langle\O_{\B}(1) \rangle$) of  
index $i(\B)=(r+\delta)/2$,  or it is a hyperplane section  of the 
Segre variety $\PP^2\times\PP^2\subset\PP^8$.
 \item\label{part: hilbert polynomial}
If $\delta>0$ and $\B$  is not a hyperplane section  
 of $\PP^2\times\PP^2\subset\PP^8$,   
putting $i=i(\B)$ and $P(t)=P_{\B}(t)$,   we have
$$
\begin{array}{c}
\begin{array}{ccc}
 P(0)=1,& P(1)=n+1, & P(2)=\left({n}^{2}+n-2\right)/2,\\
\end{array}\\
P(-1)=P(-2)=\cdots=P(-i+1)=0,\\
\forall\,t\quad P(t)=(-1)^r P(-t-i).
\end{array}
$$
In particular, if $c(\B)=r+1-i\leq5$, it remains determined $P(t)$.     
\item\label{part: cohomology} Hypothesis as in part  
 \ref{part: hilbert polynomial}. We have 
$$
 h^i(\B,\O_{\B}(t))=\left\{
\begin{array}{cccc}   0 &\mbox{if}&i=0,&t<0\\
                    P(t) &\mbox{if}&i=0,&t\geq0\\     
                     0 &\mbox{if}&0<i<r,&t\in\ZZ\\
                     (-1)^r P(t)&\mbox{if}&i=r,&t<0\\
                     0 &\mbox{if}&i=r,& t\geq0
\end{array} \right.
$$
\end{enumerate}
\end{proposition}
\begin{proof}
(\ref{part: B is linearly normal}) $\B$ is nondegenerate 
by Proposition \ref{prop: type 2-1}.
$\B$ is projectively normal if and 
only if  $h^1(\PP^n,\I_{\B,\PP^n}(k))=0$ for every $k\geq1$,    and this,
by Lemma \ref{prop: cohomology twisted ideal},
is true whenever 
$1\geq n-2\,r-1$, 
i.e. whenever
 $\delta\geq0$.  
Note that for   
 $\delta>0$ the thesis follows also from  \cite[Corollary~2]{bertram-ein-lazarsfeld}.

\noindent
(\ref{part: B is fano}) We know that
 $\B$ is a 
$QEL$-variety of type $\delta>0$.     If $\delta\geq3$ the thesis is 
contained in \cite[Theorem~2.3]{russo-qel1};  
for $0<\delta\leq2$ we apply  
\cite[Theorem~2.2]{ionescu-russo-conicconnected}.
Thus we have that
$\B$ is a Fano variety with
 $\Pic(\B)\simeq\ZZ\langle\O_{\B}(1)\rangle$,   or
 $\B$ is one of the following:
\begin{enumerate}
 \item\label{item: rational normal scroll} a rational normal scroll,
 \item a hyperplane section  of $\PP^2\times\PP^2\subset\PP^8$,
 \item\label{item: P2xP3inP11} $\PP^2\times\PP^3\subset\PP^{11}$.
\end{enumerate}
This follows only by imposing  that the pair $(r,n)$ corresponds to
that of
 a variety listed in \cite[Theorem~2.2]{ionescu-russo-conicconnected}.
Of course  case \ref{item: P2xP3inP11} is impossible 
because $h^0(\PP^{11},\I_{\B}(2))=13$ while
 $h^0(\PP^{11},\I_{\PP^2\times\PP^3}(2))=18$.
Finally, if $\B$ is as in
 case \ref{item: rational normal scroll},
since $\B$ is smooth and $n\geq 2r+1$, 
we get
 $\dim(\Sec(\B))=2r+1$ (see for example \cite[Proposition~1.5.3]{russo-specialvarieties}) 
and hence $\delta=0$ (reasoning on the number of quadrics also it results
 $r=1$, $n=4$, $d=2$, $\Delta=2$).

\noindent
(\ref{part: hilbert polynomial}) By Lemma \ref{prop: cohomology twisted ideal} 
we have
$h^j(\PP^n,\I_{\B,\PP^n}(k))=0$, for every $j>0$, $k\geq 1-\delta$,
and by the structural sequence we get
$h^j(\B,\O_{\B}(k))=0$, for every $j>0$, $k\geq 1-\delta$.
Hence, by $\delta>0$, it follows
\begin{eqnarray*}
P(0)&=&h^0(\B,\O_{\B})=1, \\
 P(1)&=&h^0(\B,\O_{\B}(1))=h^0(\PP^n,\O_{\PP^n}(1))=n+1, \\
P(2)&=&h^0(\PP^n,\O_{\PP^n}(2))-h^0(\PP^n,\I_{\B,\PP^n}(2))=\left({n}^{2}+n-2\right)/2. 
\end{eqnarray*}
Moreover, if $t<0$, by  Kodaira Vanishing Theorem  
 and  Serre Duality,
$$
P(t)=(-1)^r h^r(\B,\O_{\B}(t))=(-1)^r h^0(\B,\O_{\B}(-t-i)),
$$
and hence for $-i<t<0$ we have $P(t)=0$ 
and for $t\leq -i$ (hence   for every $t$) we have $P(t)=(-1)^r P(-t-i)$.
In particular
$$
P(-i)=(-1)^r,\quad P(-i-1)=(-1)^r (n+1),\quad P(-i-2)=(-1)^r \left({n}^{2}+n-2\right)/2,
$$
 and we have at least $i+5$ independent conditions for $P(t)$.

\noindent
(\ref{part: cohomology}) As the part (\ref{part: hilbert polynomial}).
\end{proof}

\begin{proposition}\label{Prop: numerical restrictions}  
\begin{enumerate}
\item\label{Part: first, numerical restrictions}   
If $d$ is even and   $\Delta=2$ , then  
$\delta\in\{0,1,3,7\}$,  $r=d\delta+d-1$, $n=(2d-1)\delta+2d$.
\item\label{Part: second, numerical restrictions} 
 If $d$,   $\Delta$ are both odd, then 
$\delta=0$, $r=\Delta+d-3$, $n=2( \Delta+d-2)$.
\end{enumerate}
\end{proposition} 
\begin{proof}
(\ref{Part: first, numerical restrictions})
By \cite[Theorem~2.8]{russo-qel1},
for $\delta\geq3$,   we have
$
r\equiv \delta \ \mathrm{mod}\ 2^{\left\lfloor(\delta-1)/2\right\rfloor}
$,
i.e.
\begin{equation}\label{eq: divisibility theorem}
\Delta-2+(d-1)(\delta+1)\equiv 0 \ \mathrm{mod}\ 2^{\left\lfloor(\delta-1)/2\right\rfloor}.
\end{equation}
If $d$ is even and $\Delta=2$, then (\ref{eq: divisibility theorem}) becomes
$
\delta+1\equiv 0 \ \mathrm{mod}\ 2^{\left\lfloor(\delta-1)/2\right\rfloor}
$
and we conclude that  $\delta\in\{0,1,2,3,7\}$.
Moreover, if $\delta=2$,  
by Proposition 
\ref{prop: cohomology properties}
we would get
 the contradiction that $\B$ is a Fano variety with
\begin{equation}\label{eq: contradiction by conic-connected}
 2i(\B)=r+\delta=\Delta+\left( d+1\right) \delta+d-3 \equiv 1\ \mathrm{mod}\ 2.
\end{equation}

\noindent
(\ref{Part: second, numerical restrictions})
  If $\delta>0$,  
by Proposition 
\ref{prop: cohomology properties}
we would get
again the contradiction
 that $\B$ is a Fano variety and  holds (\ref{eq: contradiction by conic-connected}).
Thus $\delta=0$ and there follow the expressions 
 of $r$ and   $n$ by Proposition \ref{prop: dimension formula}. 
\end{proof}

\section{Examples}\label{sec: examples}
The calculations in the following examples 
can be verified using \cite{macaulay2} or \cite{sagemath}. 
\begin{example}[$\Delta=2$, $d=2$]\label{example: d=2 Delta=2}
Let $\psi:\PP^{n+1}\dashrightarrow\PP^{n+1}$ be a Cremona  transformation  of 
type $(2,2)$. 
If $H\simeq\PP^{n}\subset\PP^{n+1}$ is a general hyperplane, 
then $\Q:=\overline{\psi(H)}\subset\PP^{n+1}$ is a quadric hypersurface 
 and the restriction  
$\psi|_{H}:\PP^n\dashrightarrow\Q\subset\PP^{n+1}$ is 
a birational transformation
of type $(2,2)$, with base locus 
 $
\mathrm{Bs}(\psi|_H)=\mathrm{Bs}(\psi)\cap H$.
If $\psi$ is special, i.e. if its base locus is a Severi variety, 
then also $\psi|_H$ is special and moreover
it is possible to verify that
the quadric $\Q$ is smooth, for example by determining explicitly the equation. 
\end{example}
\begin{example}[$\Delta=2$, $d=3$, $\delta=0$]\label{example: d=3 Delta=2}
We construct the Edge variety
 $\mathcal{X}$ of 
dimension $3$ and degree $7$ (see also \cite{edge} and \cite[Example~2.4]{ciliberto-mella-russo}).
Consider the Segre variety
$S^{1,3}=\PP^1\times\PP^3\subset\PP^7$ and  choose 
coordinates $x_0,\ldots,x_7$ on $\PP^7$ such that the equations 
of $S^{1,3}$ in $\PP^7$ are given by
$$
\begin{array}{ccc}Q_0=- x_{1} x_{4} + x_{0} x_{5},&Q_1= - x_{2} x_{4} + x_{0} x_{6},&Q_2= - x_{3} x_{4} + x_{0} x_{7},\\ 
                  Q_3= - x_{2} x_{5} + x_{1} x_{6},&Q_4= - x_{3} x_{5} + x_{1} x_{7},&Q_5= - x_{3} x_{6} + x_{2} x_{7}.
\end{array}
$$
Take a general quadric 
 $V(Q)$ containing the linear space 
$P=V(x_0,x_1,x_2,x_3)\subset S^{1,3},$
i.e. $Q=\sum b_{ij}x_i x_j$, for suitable coefficients $b_{ij}\in\CC$, with 
$i\leq j$,  $0\leq i\leq 3$ and  $0\leq j \leq 7$.
The intersection $\mathcal{Y}=S^{1,3}\cap V(Q)$ is an equidimensional variety
of dimension $3$ that has $P$ as irreducible component of multiplicity $1$.
Hence it  defines a variety
 $\mathcal{X}$ of dimension $3$ and degree $7$
 such that
$\mathcal{Y}=P\cup \mathcal{X}$ and  $P\nsubseteq \mathcal{X}$.
Since we are interested in
 $\mathcal{Y}$ and not in $Q$,
we can suppose 
$
b_{14}=b_{24}=b_{34}=b_{25}=b_{35}=b_{36}=0
$.
Thus it is easy to verify that
 $\mathcal{X}$ is the scheme-theoretic intersection of  $\mathcal{Y}$ with  the quadric $V(Q')$,   where
 \begin{eqnarray*}
  Q'&=&b_{37}x_7^2+b_{27}x_6x_7+b_{17}x_5x_7+b_{07}x_4x_7+b_{33}x_3x_7  +  b_{26}x_6^2+b_{16}x_5x_6 + \\ &+& 
       b_{06}x_4x_6+b_{23}x_3x_6+b_{22}x_2x_6 +b_{15}x_5^2+b_{05}x_4x_5+b_{13}x_3x_5+b_{12}x_2x_5 + \\ &+&
       b_{11}x_1x_5  +  b_{04}x_4^2+b_{03}x_3x_4+b_{02}x_2x_4+b_{01}x_1x_4+b_{00}x_0x_4.
 \end{eqnarray*}
The rational map
$\psi:\PP^7\dashrightarrow\PP^7$,
defined by $\psi([x_0,\ldots,x_7])=[Q_0,\ldots,Q_5,Q,Q']$,
has image the 
 rank $6$ quadric
$\Q=V(y_{0} y_{5} - y_{1} y_{4} + y_{2} y_{3})$
and the closure of its general fiber is a $\PP^1\subset\PP^6$. 
Hence restricting 
$\psi$ to a general $\PP^6\subset\PP^7$, we get a birational transformation  
$
\PP^6\dashrightarrow\Q\subset\PP^7$.  
\end{example}
\begin{example}\label{example: d=3 Delta=2 continuing} Continuing with the Example
 \ref{example: d=3 Delta=2},
we set
$b_{00}=b_{07}=b_{15}=b_{16}=b_{22}=b_{33}=1$ and for all the other indices
 $b_{ij}=0$.
Substitute in the quadrics
 $Q_0,\ldots,Q_5,Q,Q'$ instead of the variable $x_7$ 
the variable $x_0$, i.e. we consider the intersection of $\mathcal{X}$ 
with the hyperplane $V(x_7-x_0)$.
Denote by 
 $X\subset\PP^6$ the scheme so obtained.
$X$ is irreducible, smooth, it is defined by the $8$ independent quadrics:
\begin{equation}\label{eq: equations of B, d=3 Delta=2}
\begin{array}{c}
x_3x_6-x_0x_2,\ x_5x_6+x_2x_6+x_5^2+2x_0x_4+x_0x_3,\ x_3x_5-x_0x_1,\ x_2x_5-x_1x_6,\\
x_3x_4-x_0^2,\ x_2x_4-x_0x_6,\ x_1x_4-x_0x_5,\ x_1x_6+x_1x_5+x_3^2+x_2^2+2x_0^2,
\end{array}
\end{equation}
and its Hilbert polynomial is
 $P_{X}(t)=(7t^2+5t+2)/2$.
The quadrics (\ref{eq: equations of B, d=3 Delta=2}) 
define a birational map 
 $\psi:\PP^6\dashrightarrow\Q\subset\PP^7$ into the quadric 
 $\Q=V(y_0y_6-y_2y_5+y_3y_4)$ and the inverse of $\psi$
 is defined by the cubics:
\begin{equation}\label{eq: inverse, d=3 Delta=2}
\begin{array}{c}
y_0y_5y_7-y_1y_4y_7+y_1y_2y_6+y_0y_1y_6+y_2y_3y_5 + y_0y_3y_5+y_2^2y_4+y_0y_2y_4, \\
-y_6y_7^2-2y_4y_6y_7-y_3y_6y_7+y_0y_3y_7-y_1y_2y_7+2y_2y_6^2+2y_0y_6^2+y_2y_3^2  +y_0y_3^2+y_2^3+y_0y_2^2, \\
-y_5y_7^2-2y_4y_5y_7-y_3y_5y_7-y_0y_1y_7+2y_2y_5y_6+2y_0y_5y_6-y_1y_2y_3-y_0y_1y_3+y_0y_2^2+y_0^2y_2, \\
-y_4y_7^2+y_0y_6y_7-y_2y_5y_7-2y_4^2y_7-y_0^2y_7+2y_2y_4y_6+2y_0y_4y_6-y_0y_2y_3-y_0^2y_3-y_1y_2^2   -y_0y_1y_2, \\
-y_5^2y_7-y_4^2y_7-y_1y_6^2-y_3y_5y_6-y_1y_5y_6-y_2y_4y_6-2y_4y_5^2-y_3y_5^2-y_2y_4y_5-2y_4^3-y_1^2y_4   -y_0^2y_4, \\
-y_1y_6y_7-y_3y_5y_7-y_2y_4y_7-y_1y_3y_6+y_0y_2y_6-2y_2y_5^2-y_3^2y_5-y_2^2y_5-2y_2y_4^2-y_1^2y_2-y_0^2y_2, \\
-y_1y_5y_7-y_0y_4y_7+y_1y_3y_6-y_0y_2y_6-2y_0y_5^2+y_3^2y_5+y_2^2y_5 -2y_0y_4^2-y_0y_1^2-y_0^3 .
\end{array}
\end{equation}
The base locus  
 $Y\subset\Q\subset\PP^7$ of $\psi^{-1}$ 
is obtained by intersecting the scheme defined by
 (\ref{eq: inverse, d=3 Delta=2}) with the quadric  $\Q$.
$Y$ is irreducible with Hilbert polynomial  
$P_{Y}(t)=(9t^4+38t^3+63t^2+58t+24)/4!$
and its singular locus has dimension  $0$.
\end{example}
\begin{example}[$\Delta=2$, $d=4$, $\delta=0$]\label{example: d=4 Delta=2} 
Using a known construction (see \cite[\Rmnum{4} Theorem~4.5]{zak-tangent}) we can 
 determine the equations of the spinorial variety  $S^{10}\subset\PP^{15}$, which are:
\begin{equation}\label{eq: equations of spinorial S10}
\begin{array}{cc} 
 x_{7} x_{8} - x_{6} x_{9} + x_{5} x_{10} - x_{3} x_{15}, &x_{1} x_{6} + x_{4} x_{7} - x_{2} x_{10} + x_{3} x_{13},\\
 - x_{4} x_{5} + x_{0} x_{6} - x_{2} x_{8} + x_{3} x_{11},& x_{9} x_{11} - x_{8} x_{12} + x_{5} x_{14} - x_{0} x_{15}, \\
 x_{10} x_{12} - x_{9} x_{13} + x_{7} x_{14} + x_{1} x_{15},& - x_{1} x_{11} - x_{4} x_{12} + x_{0} x_{13} - x_{2} x_{14},\\
 x_{10} x_{11} - x_{8} x_{13} + x_{6} x_{14} - x_{4} x_{15},& x_{1} x_{5} + x_{0} x_{7} - x_{2} x_{9} + x_{3} x_{12}, \\
 - x_{7} x_{11} + x_{6} x_{12} - x_{5} x_{13} + x_{2} x_{15}, & x_{1} x_{8} + x_{4} x_{9} - x_{0} x_{10} + x_{3} x_{14}.
\end{array}
\end{equation}
The $10$ quadrics (\ref{eq: equations of spinorial S10})  define a rational map  
 $\psi:\PP^{15}\dashrightarrow\PP^{9}$, with base locus  
 $S^{10}$ and image in $\PP^9$
the smooth quadric 
$$\Q=V({y_8}{ y_9}-{ y_6}{ y_7}-{y_0}{ y_5}+ { y_2}{ y_4}+{ y_1}{ y_3}).$$
By \cite[page~798]{ein-shepherdbarron} the closure of 
the general fiber of  $\psi$ is
 a $\PP^7\subset\PP^{15}$   and hence,
by restricting 
 $\psi$ to a general  $\PP^8\subset\PP^{15}$, 
 we get a special birational transformation 
$
\PP^8\dashrightarrow\Q 
$
(necessarily of type $(2,4)$, by Remark \ref{remark: dimension formula without hypothesis}).
\end{example}
\begin{example}[$\Delta=3$, $d=3$]\label{example: d=3 Delta=3}
Let
$u:\PP^3\stackrel{\simeq}{\dashrightarrow}Q=V(u_0\,u_4-u_1^2-u_2^2-u_3^2)\subset\PP^4$
be defined by
$u([z_0,z_1,z_2,z_3]) =[{z}_{0}^{2},{z}_{0}{z}_{1},{z}_{0}{z}_{2},{z}_{0}{z}_{3},{z}_{1}^{2}+{z}_{2}^{2}+{z}_{3}^{2}]$.
Consider the composition
$ v^0:\PP^3\stackrel{u}{\dashrightarrow}\PP^4\stackrel{\nu_2}{\longrightarrow}\PP^{14}\dashrightarrow\PP^{13} $,
where $\nu_2$ is the Veronese map (with lexicographic order) 
and the last map is the projection onto the hyperplane
 $V(v_4)\simeq\PP^{13}\subset\PP^{14}$. 
The map $v^0$ parameterizes a nondegenerate variety in $\PP^{13}$, of degree 
 $16$ and isomorphic to the quadric $Q$. 
Take the point  $p_1=[1,0,0,0]\in\PP^3$ and  consider the composition
$v^1:\PP^3\stackrel{v^0}{\dashrightarrow}\PP^{13}\stackrel{\pi_1}{\dashrightarrow}\PP^{12} 
$,
where $\pi_1$ is the projection from the point $v^0(p_1)$ (precisely,
 if $j$ is  the index of the last nonzero coordinate
of $v^0(p_1)$, 
 we exchange the coordinates 
 $v_j$,   $v_{13}$ and 
 we project onto the hyperplane $V(v_{13})$ from the point $v^0(p_1)$).
Repeat the construction with the points $p_2=[0,0,0,1]$, $p_3=[1,0,0,1]$, $p_4=[0,1,0,1]$, $p_5=[0,0,1,1]$, 
obtaining the maps  
$v^2:\PP^3\dashrightarrow\PP^{11}$, $v^3:\PP^3\dashrightarrow\PP^{10}$, $v^4:\PP^3\dashrightarrow\PP^{9}$, $v^5:\PP^3\dashrightarrow\PP^{8}$.
       The map $v^5$ is given by
       \begin{eqnarray*}
       v^5([z_0,z_1,z_2,z_3])&=&[z_0\,z_3^3-z_0^2\,z_3^2+z_0\,z_2^2\,z_3+z_0\,z_1^2\,z_3, \   
                              -z_0\,z_1\,z_3^2-z_0\,z_1\,z_2^2-z_0\,z_1^3+z_0^3\,z_1, \\
                             && -z_0\,z_2\,z_3^2-z_0\,z_2^3-z_0\,z_1^2\,z_2+z_0^3\,z_2, \  
                              z_0^3\,z_3-z_0^2\,z_3^2, \\
                             && -z_0\,z_1\,z_3^2-z_0\,z_1\,z_2^2-z_0\,z_1^3+z_0^2\,z_1^2, \
                              z_0^2\,z_1\,z_2, \  z_0^2\,z_1\,z_3, \  z_0^2\,z_2\,z_3, \\   
                             && -z_0\,z_2\,z_3^2-z_0\,z_2^3+z_0^2\,z_2^2-z_0\,z_1^2\,z_2]
       \end{eqnarray*}
       and parameterizes the smooth variety $X\subset\PP^8$ defined by the $10$ independent quadrics:
       \begin{equation}\label{eq: equations of B, d=3 Delta=3}
       \begin{array}{c}
       x_5\,x_8-x_4\,x_8+x_1\,x_8-x_4\,x_7+x_5\,x_6+x_2\,x_6+x_5^2+x_4\,x_5-x_3\,x_5-x_2\,x_5,  \\
       x_5\,x_8-x_4\,x_8-x_4\,x_7+x_5\,x_6+x_2\,x_6+x_5^2+x_4\,x_5-x_3\,x_5-x_1\,x_5+x_2\,x_4,\\
       -x_7\,x_8-x_0\,x_8-x_7^2+x_3\,x_7-x_5\,x_6+x_0\,x_2,\\
        x_6\,x_8-x_4\,x_7+x_6^2+x_5\,x_6+x_4\,x_6-x_1\,x_6-x_0\,x_6+x_3\,x_4,\\
        x_6\,x_7+x_4\,x_7-x_5\,x_6-x_2\,x_6+x_3\,x_5,\\
        x_7\,x_8+x_3\,x_8+x_7^2-x_2\,x_7-x_0\,x_7+x_5\,x_6,\\
        x_1\,x_7-x_2\,x_6,\\
        x_5\,x_7+x_6^2+x_4\,x_6-x_1\,x_6-x_0\,x_6+x_3\,x_4,\\
       x_2\,x_6-x_3\,x_5+x_0\,x_5,\\
       x_3\,x_6-x_1\,x_6-x_0\,x_6+x_3\,x_4-x_0\,x_4+x_0\,x_1. 
       \end{array}
       \end{equation}
The quadrics (\ref{eq: equations of B, d=3 Delta=3}) 
define a special birational 
transformation 
 $\psi:\PP^8\dashrightarrow\sS\subset\PP^9$ of type $(2,3)$ 
   into the cubic $\sS\subset\PP^9$
defined by
\begin{equation}
\begin{array}{l}
 - y_7\,y_8\,y_9 + y_6\,y_8\,y_9 + y_3\,y_8\,y_9 + y_7^2\,y_9 - y_6\,y_7\,y_9 - 2\,y_3\,y_7\,y_9 - y_2\,y_7\,y_9 +\\
+ y_0\,y_7\,y_9 + y_4\,y_6\,y_9 + y_3\,y_6\,y_9 - y_4\,y_5\,y_9 - y_1\,y_5\,y_9 - y_4^2\,y_9 - y_0\,y_4\,y_9  +\\
+ y_3^2\,y_9 + y_2\,y_3\,y_9 - y_0\,y_3\,y_9 + y_7\,y_8^2 - y_3\,y_8^2 - y_6\,y_7\,y_8 + 2\,y_4\,y_7\,y_8  +\\
+ y_3\,y_7\,y_8 - y_2\,y_7\,y_8 - y_4\,y_6\,y_8 + y_2\,y_6\,y_8 - y_0\,y_6\,y_8 + y_4\,y_5\,y_8 + y_0\,y_5\,y_8 +\\
- 3\,y_3\,y_4\,y_8 - y_1\,y_4\,y_8 + y_0\,y_4\,y_8 - y_3^2\,y_8 + y_2\,y_3\,y_8 - y_1\,y_3\,y_8 + y_5\,y_7^2  +\\
+ y_2\,y_7^2 + y_1\,y_7^2 - y_0\,y_7^2 + y_6^2\,y_7 + y_3\,y_6\,y_7 + y_0\,y_6\,y_7 - y_3\,y_5\,y_7 +\\
+ y_1\,y_5\,y_7 - y_2\,y_3\,y_7 - y_1\,y_3\,y_7 + y_0\,y_3\,y_7 + y_0\,y_2\,y_7 - y_4\,y_6^2 - y_3\,y_6^2 +\\
- y_3\,y_5\,y_6 - y_3\,y_4\,y_6 - y_0\,y_4\,y_6 - y_3^2\,y_6 - y_0\,y_3\,y_6 - y_1\,y_2\,y_6 + y_4^2\,y_5 +\\
+ y_3\,y_4\,y_5 + y_0\,y_4\,y_5 + y_2\,y_4^2 - y_1\,y_4^2 + y_0\,y_4^2 + y_2\,y_3\,y_4 - y_1\,y_3\,y_4 +\\
+ y_0\,y_3\,y_4 + y_1\,y_2\,y_4 .
\end{array}
\end{equation}
The singular locus of $\sS$ has dimension $3$, 
from which it follows the factoriality of $\sS$ 
(see  \cite[\Rmnum{11} Corollaire~3.14]{sga2}).
\end{example}
\begin{example}[$\Delta=4$, $d=2$, $\delta=0$]\label{example: B2=singSred-3fold} 
This is an example for which (\ref{eq: hypothesis on sing locus hypersurface}) is not satisfied;
essentially, it gives an example for case (\ref{part: case 0.5, 3-fold}) of 
Proposition \ref{prop: 3-fold in P8 - S nonlinear}.
Consider the irreducible
smooth $3$-fold $X\subset\PP^8$ defined as the intersection of
 $\PP^2\times\PP^2\subset\PP^8$, which is defined by
\begin{equation}\label{eq: P2xP2inP8}
\begin{array}{c}
 x_5x_7-x_4x_8,\ 
x_2x_7-x_1x_8,\ 
x_5x_6-x_3x_8,\  
x_4x_6-x_3x_7,\ 
x_2x_6-x_0x_8,\\ 
x_1x_6-x_0x_7,\ 
x_2x_4-x_1x_5,\ 
x_2x_3-x_0x_5,\ 
x_1x_3-x_0x_4,
\end{array}
\end{equation}
 with the quadric hypersurface defined by:
 \begin{equation}
2x_0^2+3x_1^2+5x_2^2+x_3^2+x_4^2+x_5^2+5x_6^2+3x_7^2+2x_8^2.
 \end{equation}
Thus $X$ is defined by $10$ quadrics and we can consider the associated 
rational map $\psi:\PP^8\dashrightarrow\sS=\overline{\psi(\PP^8)}\subset\PP^9$.
We have that $\sS$ is the quartic hypersurface defined by:
 \begin{equation}
\begin{array}{l}
  2y_1^2y_2^2+3y_1^2y_3^2-4y_0y_1y_2y_4+2y_0^2y_4^2+5y_3^2y_4^2-6y_0y_1y_3y_5-10y_2y_3y_4y_5+3y_0^2y_5^2 \\ +5y_2^2y_5^2+y_2^2y_6^2+y_3^2y_6^2+5y_4^2y_6^2+3y_5^2y_6^2-2y_0y_2y_6y_7-10y_1y_4y_6y_7+y_0^2y_7^2 \\ +5y_1^2y_7^2+y_3^2y_7^2+2y_5^2y_7^2-2y_0y_3y_6y_8-6y_1y_5y_6y_8-2y_2y_3y_7y_8-4y_4y_5y_7y_8+y_0^2y_8^2 \\ +3y_1^2y_8^2+y_2^2y_8^2+2y_4^2y_8^2-y_3y_4y_6y_9+y_2y_5y_6y_9+y_1y_3y_7y_9-y_0y_5y_7y_9-y_1y_2y_8y_9 \\ +y_0y_4y_8y_9,  
\end{array}
 \end{equation}
and $\psi$ is birational with inverse defined by:
\begin{equation}
\begin{array}{c}
 y_5y_7-y_4y_8,\ 
y_5y_6-y_1y_8,\  
y_4y_6-y_1y_7,\  
y_3y_7-y_2y_8,\  
y_3y_6-y_0y_8, \\
y_2y_6-y_0y_7,\  
y_3y_4-y_2y_5,\  
y_1y_3-y_0y_5,\  
y_1y_2-y_0y_4.
\end{array}
\end{equation}
Moreover $X$ is a Mukai variety of degree $12$, sectional genus $7$ and Betti numbers
$b_2=2$, $b_3=18$. 
The secant variety $\Sec(X)$ 
is a cubic hypersurface and for the general point $p\in\Sec(X)$
there are two secant lines of $X$ through $p$.
Finally, denoting by $Y\subset\sS$ the base locus of $\psi^{-1}$, we have 
$\dim(Y)=5$, $\dim(\sing(Y))=0$ and
$
Y=(Y)_{\mathrm{red}}=(\sing(\sS))_{\mathrm{red}}
$.
\end{example}
\section{Transformations of type \texorpdfstring{$(2,2)$}{(2,2)} into a quadric}\label{sec: type 2-2 into quadric} 
\begin{theorem}\label{prop: classification of type 2-2 into quadric}
If $\Delta=2$, $\varphi:\PP^n\dashrightarrow \Q:=\sS\subset\PP^{n+1}$
is of type $(2,2)$ and $\Q$ is smooth,
then 
$\B$ is a hyperplane section of a Severi variety.
\end{theorem}
\begin{proof}
By Proposition \ref{Prop: numerical restrictions} part \ref{Part: first, numerical restrictions} we get
 $\delta\in\{0,1,3,7\}$.
If $\delta=0$,
we have $r=1$,   $n=4$  and the 
thesis follows from Proposition \ref{prop: P4}. 
Alternatively, we can apply 
 Lemma \ref{prop: cohomology twisted ideal} 
to  determine the Hilbert polynomial of the curve  $\B$.
If $\delta=1$,
we have $r=3$,   $n=7$ and,
by  Proposition \ref{prop: cohomology properties} 
part \ref{part: B is fano},
 $\B$ is a hyperplane section of the Segre variety 
 $\PP^2\times\PP^2\subset\PP^8$,  or 
it is a Fano variety of the first species of     
index $i(\B)=r-1$.
The latter case cannot occur
by the classification of Del Pezzo varieties
in \cite[\Rmnum{1} \S  8]{fujita-polarizedvarieties}. 
If $\delta=3$,
we have $r=7$,   $n=13$  
and $\B$ is a Mukai variety of the first species.
   By Proposition
 \ref{prop: cohomology properties} part \ref{part: hilbert polynomial}, 
we can 
determine the Hilbert polynomial  of $\B$ and in particular  
to get that the sectional genus is $g=8$.     
Hence, applying
 the classification of Mukai varieties in \cite{mukai-biregularclassification},
we get that
$\B$ is 
a hyperplane section 
 of the Grassmannian $\GG(1,5)\subset\PP^{14}$.
Now, suppose
 $\delta=7$ and hence 
$r=15$, $n=25$, $r'=16$. 
Keep the notation as
 in the proof of Proposition \ref{prop: dimension formula} and  put
$Y=(\B')_{\mathrm{red}}$. We shall show that $\B\subset\PP^{25}$ 
is a hyperplane section of $E_6\subset\PP^{26}$ in several steps.
\begin{step}
$\pi'(E)=\left(\Sec(\B')\cap\Q\right)_{\mathrm{red}}=\left(\Sec(Y)\cap\Q\right)_{\mathrm{red}}$ and for the general point $p\in\pi'(E)$ 
 the entry locus $\Sigma_p(\B')$ is a quadric of dimension $8$.
\end{step}
\begin{proof}[of the Claim]
See also \cite[Proposition~2.3]{ein-shepherdbarron}.
If $x\in\B$ is a point, then $\pi'(\pi^{-1}(x))\simeq\PP^9$ and by the relations (\ref{eq: EE'}) 
we get that
$\pi'(\pi^{-1}(x))\cap\B'$  is a quadric hypersurface in $\PP^9$. 
Then, for every point $p\in \pi'(\pi^{-1}(x))\setminus \B'$, every line contained
in $\pi'(\pi^{-1}(x))$ and through  
 $p$ is a secant line of $\B'$  and therefore
 it is contained in $L_p(\B')$. Thus 
$$\pi'(\pi^{-1}(x))\subseteq L_p(\B')\subseteq \Sec(\B')\cap\Q$$
and varying $x\in\B$ we get
\begin{equation}\label{eq: first inclusion, type 2-2}
\pi'(E)=\bigcup_{x\in \B} \pi'(\pi^{-1}(x))\subseteq \Sec(\B')\cap \Q.
\end{equation}
Conversely,
if
 $p\in\Sec(\B')\cap\Q\setminus\B'$, 
then every secant line  
 $l$ of $\B'$ through  $p$ is contained 
 in $\Q$ and $\varphi^{-1}(l)=\varphi^{-1}(p)=:x$. 
Hence the strict transform $\widetilde{l}=\overline{\pi'^{-1}(l\setminus\B')}$ 
is contained
in $\pi^{-1}(x)\subseteq E$ 
 and then $l$ is contained in $\pi'(\pi^{-1}(x))\subseteq\pi'(E)$.
Thus 
$$L_p(\B')\subseteq \pi'(\pi^{-1}(x))\subseteq\pi'(E)$$
 and varying $p$ 
(since $p$ lies on at least a secant line of $\B'$)
we get 
\begin{equation}\label{eq: second inclusion, type 2-2}
\left(\Sec(\B')\cap\Q\right)_{\mathrm{red}} = \overline{\left(\Sec(\B')\cap\Q\setminus\B'\right)_{\mathrm{red}}}\subseteq \pi'(E).
\end{equation}
The conclusion follows from (\ref{eq: first inclusion, type 2-2}) and (\ref{eq: second inclusion, type 2-2}) and 
by observing that, since $\B'$ is irreducible and generically reduced, we have
 $(\Sec(\B'))_{\mathrm{red}}=\Sec(Y)$.
\end{proof} 
\begin{step}\label{step: tang proj and entry locus}
For the general point $x\in\B$, 
$\overline{\tau_{x,\B}(\B)}=W_{x,\B}\simeq \Sigma_{p}(\B')$, 
with $p$ general point in $\pi'(\pi^{-1}(x))$.
\end{step}
\begin{proof}[of the Claim]
See also \cite[Theorem~1.4]{mella-russo-baselocusleq3}.
The projective space $\PP^9\subset\PP^{25}$ skew to $T_x(\B)$ (i.e.
the codomain of $\tau_{x,\B}$) identifies with
 $\pi^{-1}(x)=\PP({\mathcal{N}}_{x,\B}^{\ast})$.
For $z\in\B\setminus T_x(\B)$, the line 
$l=\langle z,x \rangle$ 
corresponds to a point $w\in \pi^{-1}(x)=\PP({\mathcal{N}}_{x,\B}^{\ast})$.
Moreover $\pi'(w)=\varphi(l)\in \B'$ and 
 $\tau_{x,\B}(z)=\tau_{x,\B}(l)=w$, so that 
$\pi'(w)\in \pi'(\pi^{-1}(x))\cap \B'$
and hence $W_{x,\B}$ identifies with an irreducible component of
 $\pi'(\pi^{-1}(x))\cap \B'$.
Since $W_{x,\B}$ is a nondegenerate hypersurface  
in $\PP({\mathcal{N}}_{x,\B}^{\ast})$ (see \cite[Proposition~1.3.8]{russo-specialvarieties}), 
we have 
\begin{equation}
W_{x,\B}\simeq\pi'(\pi^{-1}(x))\cap \B'\subset\pi'(\pi^{-1}(x))\simeq\PP^9.
\end{equation}
\end{proof} 
\begin{step}\label{step: smooth image tangential projection}
For the general point $x\in\B$, $W_{x,\B}\subset\PP^9$ is 
a smooth quadric hypersurface.
\end{step}
\begin{proof}[of the Claim]
By \cite[Theorem~2.3]{russo-qel1}
it follows that  $\L_{x,\B}\subset\PP^{14}$ is a smooth 
$QEL$-variety of dimension $\dim(\L_{x,\B})=9$ and of
type $\delta(\L_{x,\B})=5$.
So, applying \cite[Corollary~3.1]{russo-qel1}, we get that
$\L_{x,\B}\subset\PP^{14}$ is projectively equivalent to a hyperplane section 
 of the spinorial variety $S^{10}$.
 Hence,   by \cite[Proposition~2.3.2]{russo-specialvarieties}
and Example \ref{example: d=4 Delta=2},
it follows that  $W_{x,\B}\subset\PP^9$ is
a smooth quadric  of dimension $r-\delta=8$. 
\end{proof} 
Summing up the previous claims, we obtain that for the general point
$p\in\pi'(E)\setminus Y$ the entry locus $\Sigma_p(\B')$ is a smooth quadric;
moreover,   one can be easily convinced that we also have
$\Sigma_{p}(\B')=L_p(\B')\cap \B'=L_p(\B')\cap Y=L_p(Y)\cap Y = \Sigma_p(Y)$.
Now, by the proof of  Proposition \ref{prop: dimension formula}, 
it follows that $\pi'(E)$ is a reduced and irreducible divisor 
in $|\O_{\Q}(3)|$ and,
since $\pi'(E)=\left(\Sec(Y)\cap \Q\right)_{\mathrm{red}}$, 
we have that $\Sec(Y)$ is a hypersurface of degree a multiple of $3$. In particular
$Y\subset\PP^{26}$ is an irreducible nondegenerate variety
 with secant defect $\delta(Y)=8$.
Note also that $Y$ is different from a cone: if
 $z_0$ is a vertex of $Y$,
since $I_{\B',\PP^{26}}\subseteq I_{Y,\PP^{26}}$, 
we have that $z_0$ is a vertex of every quadric defining
 $\B'$, in particular $z_0$ is a vertex of $\B'$ and $\Q$;
then, for a general point $z\in\Q$, we have 
$\langle z, z_0 \rangle\subseteq\overline{(\varphi^{-1})^{-1}(\varphi^{-1}(z))}$, 
against the birationality  of $\varphi^{-1}$. 
\begin{step}\label{step: Y is QEL}
$Y$ is a $QEL$-variety. 
\end{step}
\begin{proof}[of the Claim]
Consider the rational map
$\psi:\Sec(Y)\dashrightarrow \PP^{26}$, 
defined by the linear system
 $|\I_{\B',\PP^{26}}(2)|$ restricted to $\Sec(Y)$. 
Since $\overline{\psi(\Sec(Y)\cap \Q)}=\overline{\varphi^{-1}(\pi'(E))}=\B$ and 
the image of $\psi$ is 
of course nondegenerate, it  follows that the dimension
of the image of
 $\psi$ is at least $16$ and hence that the dimension 
of its general fiber is at most $9$. 
Now, if $q\in \Sec(Y)\setminus Y$ is a general point, denote by
 $\widetilde{\psi^{-1}(\psi(q))}$ the irreducible component
 of $\overline{\psi^{-1}(\psi(q))}$ through
  $q$ and by $\widetilde{\Sigma_q(Y)}$ an any irreducible component
 of $\Sigma_q(Y)$; note
 that, from generic smoothness \cite[\Rmnum{3} Corollary~10.7]{hartshorne-ag}, it follows 
that $\overline{\psi^{-1}(\psi(q))}$ 
(and at fortiori $\widetilde{\psi^{-1}(\psi(q))}$) is smooth in its general point  $q$.  
We have $S(q,\widetilde{\Sigma_q(Y)})\subseteq \widetilde{\psi^{-1}(\psi(q))}$, and since 
$\dim(\widetilde{\Sigma_q(Y)})=8$, 
it follows $S(q,\widetilde{\Sigma_q(Y)})= \widetilde{\psi^{-1}(\psi(q))}$. 
Thus the cone $S(q,\widetilde{\Sigma_q(Y)})$ is smooth in its vertex and  
necessarily it follows  
$S(q,\widetilde{\Sigma_q(Y)})=\PP^9$ and $\langle \Sigma_q(Y) \rangle=\PP^9$.
Finally, by  Trisecant Lemma \cite[Proposition~1.3.3]{russo-specialvarieties}, it
follows that $\Sigma_q(Y)\subset\PP^9$ is a quadric hypersurface.
\end{proof} 
\begin{step}\label{step: gammaY=deltaY} 
$\widetilde{\gamma}(Y)=0$.
\end{step}  
\begin{proof}[of the Claim]
We first show that 
for the general point $q\in\Sec(Y)\setminus Y$ 
the entry locus $\Sigma_q(Y)$ is smooth,  by discussing two cases:
\begin{case}[Suppose $\pi'(E)\nsubseteq \sing(\Sec(Y))$]
Denote by
 $\Hilb(Y)$ the Hilbert scheme of  $8$-dimensional quadrics contained in $Y$
and by $V$ a nonempty open set of $\Sec(Y)\setminus Y$ such that for every  
$q\in V$ we have $\Sigma_q(Y)\in \Hilb(Y)$. 
If $\varrho:Y\times \left(Y\times\PP^{26}\right)\longrightarrow Y\times\PP^{26}$ is the projection, 
at the closed subscheme of $Y\times V$,
\begin{displaymath}
 \varrho\left(\overline{\{(w,z,q)\in Y\times Y\times \PP^{26}: 
w\neq z\mbox{ and } q\in\langle w,z\rangle\}}\right)\cap Y\times V    
=  \{ (z,q)\in Y\times V : z\in \Sigma_q(Y) \},
\end{displaymath}
corresponds a rational map 
$\nu:\Sec(Y)\dashrightarrow \Hilb(Y)$ which sends the point
 $q\in V$ to the quadric $\nu(q)=\Sigma_q(Y)$;
denote by  
$\mathrm{Dom}(\nu)$ the largest open set  of $\Sec(Y)$ where
 $\nu$ can be defined.
By assumption we have
 $D':=\pi'(E)\cap\reg(\Sec(Y))\neq\emptyset$.
It follows that the rational map 
$\nu':=\nu|_{\reg(\Sec(Y))}:\reg(\Sec(Y))\dashrightarrow \Hilb(Y)$,
having indeterminacy locus of codimension $\geq 2$, is defined in the general point 
 of $D'$, i.e.
\begin{equation}\label{eq: Dom intersects divisor}
\emptyset\neq\mathrm{Dom}(\nu')\cap D'=\mathrm{Dom}(\nu)\cap\pi'(E)\cap \reg(\Sec(Y))\subseteq \mathrm{Dom}(\nu)\cap\pi'(E).
\end{equation}
Now, consider the natural map
$\rho:\Hilb(Y)\longrightarrow\GG(9,26)$, defined by $\rho(Q)=\langle Q \rangle$
and  the closed subset 
$C:=\{ (q,L)\in \PP^{26}\times \GG(9,26): q\in L  \}$ 
of $\PP^{26}\times \GG(9,26)$.
We have
\begin{eqnarray*}
&& \left(\left(\mathrm{Id}_{\PP^{26}}\times\rho\right)^{-1}(C)\cap \mathrm{Dom}(\nu)\times \Hilb(Y)\right)\bigcap \Graph(\nu:\mathrm{Dom}(\nu)\to\Hilb(Y)) \\
&=&\left\{(q,Q)\in \mathrm{Dom}(\nu)\times \Hilb(Y): Q=\nu(q)\mbox{ and }q\in\langle Q\rangle \right\} \\
&\simeq& \left\{ q\in \mathrm{Dom}(\nu): q\in \langle \nu(q)\rangle\right\}=:T,
\end{eqnarray*}
from which it follows that the set  $T$ 
is closed in $\mathrm{Dom}(\nu)$ and,
since $\mathrm{Dom}(\nu)\supseteq T\supseteq V$, we have 
\begin{equation}\label{eq: property of Dom}
\mathrm{Dom}(\nu)=\overline{T}=T=\left\{ q\in \mathrm{Dom}(\nu): q\in \langle \nu(q)\rangle\right\}.
\end{equation}
By (\ref{eq: Dom intersects divisor}) and (\ref{eq: property of Dom}) 
it follows that, for the general point $p\in\pi'(E)\setminus Y$, we have $p\in\langle\nu(p)\rangle$, 
hence $\langle\nu(p)\rangle\subseteq L_p(Y)$  and then $\nu(p)=\Sigma_p(Y)$.
Thus $\nu(\mathrm{Dom}(\nu))$ intersects the open  $U(Y)$ of 
$\Hilb(Y)$ consisting of the smooth $8$-dimensional quadrics 
 contained in $Y$ and then,  for a general point 
 $q$ in the nonempty open set $\nu|_{\mathrm{Dom}(\nu)}^{-1}(U(Y))$, 
the entry locus $\Sigma_q(Y)$ is a smooth $8$-dimensional quadric.
\end{case}
\begin{case}[Suppose $\pi'(E)\subseteq \sing(\Sec(Y))$]\footnote{Note that 
this is the only place in the proof where we use the smoothness of $\Q$.
Note also that 
this case does not arise if one first shows that $\deg(\Sec(Y))=3$.
Indeed,  let $\mathcal{C}\subset\PP^N$ be a cubic hypersurface
 and $D\subset \mathcal{C}$ an irreducible divisor contained in $\sing(\mathcal{C})$.
Then, for a general plane $\PP^2\subset\PP^N$, putting $C=\mathcal{C}\cap \PP^2$ and $\Lambda=D\cap \PP^2$,
we have $\Lambda\subseteq \sing(C)$. Hence, by   B\'ezout's Theorem, being $C$ 
a plane cubic curve, it follows
 $\deg(D)=\#(\Lambda)=1$.} Let $z\in Y$ be a general point,
 $\tau_{z,Y}:Y\dashrightarrow W_{z,Y}\subset \PP^{9}$ 
the tangential projection (note that $W_{z,Y}$ is a nonlinear hypersurface, by
\cite[Proposition~1.3.8]{russo-specialvarieties}) and 
let $q\in \Sec(Y)\setminus Y$ be a general point. 
Arguing as in \cite[Claim~8.8]{chiantini-ciliberto} we obtain that 
$T_z(Y)\cap L_q(Y)=T_z(Y)\cap \langle \Sigma_q(Y) \rangle=\emptyset$, so we deduce 
that $\tau_{z,Y}$ isomorphically maps $\Sigma_q(Y)$ to $W_{z,Y}$. 
From this and Terracini Lemma it follows that 
\begin{equation}\label{eq: identity of joints}
\Sec(Y)=S(\Sigma_q(Y),Y):=\overline{\bigcup_{(z_1,z_2)\in \Sigma_q(Y)\times Y \atop  z_1\neq z_2} \langle z_1, z_2\rangle} .
\end{equation}
Now suppose by contradiction that there exists $w\in \mathrm{Vert}(\Sigma_q(Y))$.
By (\ref{eq: identity of joints}) it follows that $w\in \mathrm{Vert}(\Sec(Y))$ and 
by assumption we obtain $\pi'(E)\subseteq S(w,\pi'(E))\subseteq\sing(\Sec(Y))$.
Thus $w\in\mathrm{Vert}(\pi'(E))$ and hence $T_w(\pi'(E))=\PP^{26}$. This yields that
$\Q$ is singular in $w$, a contradiction.
\end{case}
Finally, since for general points $z\in Y$ and $q\in\Sec(Y)\setminus Y$, we have 
$W_{z,Y}=\overline{\tau_z(Y)}\simeq\Sigma_q(Y)$ and $\Sigma_q(Y)$ is a smooth quadric, 
we deduce that $W_{z,Y}$ is smooth.
It follows that the Gauss map 
$G_{W_{z,Y}}:W_{z,Y}\dashrightarrow\GG(8, 9)=\left(\PP^{9}\right)^{\ast}$ 
is birational onto its image (see \cite{zak-tangent}) and hence the dimension 
of its general fiber is $\widetilde{\gamma}(Y)=\gamma(Y)-\delta(Y)=0$.
\end{proof}
Now, consider two  general points $z_1,z_2\in Y$ ($z_1\neq z_2$),
a general point 
 $q\in\langle z_1, z_2\rangle$ and put
 $\Sigma_q=\Sigma_q(Y)$, $L_q=L_q(Y)$, $H_q=T_q(\Sec(Y))$.
Let $\pi_{L_q}:Y\dashrightarrow\PP^{16}$ be the linear projection from $L_q$ and
$H'_q\simeq\PP^{25}\subset\PP^{16}$ the projection of $H_q$ from $L_q$.
We note that,  by the proof of the
Scorza Lemma in \cite{russo-specialvarieties}, it follows that 
$\Sigma_q=\overline{\tau_{z_1,Y}^{-1}(\tau_{z_1,Y}(z_2))}$ and,
 since the image
 of the general tangential projection  is smooth of 
dimension $8$ and the general entry locus is smooth 
of dimension $8=\dim(Y)-\dim(W_{z_1,Y})$, it follows
 that $Y$ is smooth along $\Sigma_q(Y)$.
In particular,  it 
follows that $H_q$ is tangent to
$Y$ along $\Sigma_q$. 
\begin{step}\label{step: birationality}
$\pi_{L_q}$ is birational.
\end{step}
\begin{proof}[of the Claim]
See also \cite[Proposition~3.3.14]{russo-specialvarieties}.
By the generality of the points
 $z_1,z_2,q$ it follows 
\begin{equation}\label{eq: 1, step birationality}
 L_q=\langle T_{z_1}(Y),z_2 \rangle \cap \langle T_{z_2}(Y),z_1 \rangle.
\end{equation}
The projection from the linear space
 $\langle T_{z_1}(Y),z_2 \rangle$ can be obtained as the 
composition of the tangential projection 
$\tau_{z_1,Y}:Y\dashrightarrow W_{z_1,Y}\subset\PP^9$ and 
of the projection of $W_{z_1,Y}$ from the point $\tau_{z_1,Y}(z_2)$.
 Thus the projection from $\langle T_{z_1}(Y),z_2 \rangle$, 
$\pi_{z_1,z_2}:Y\dashrightarrow\PP^8$ is dominant and 
 for the general point $z\in Y$ we get 
\begin{equation}\label{eq: 2, step birationality}
  \langle T_{z_1}(Y),z_2, z \rangle \cap Y \setminus \langle T_{z_1}(Y),z_2 \rangle 
= \pi_{z_1,z_2}^{-1}(\pi_{z_1,z_2}(z))=Q_{z_1,z}\setminus \langle T_{z_1}(Y),z_2 \rangle,
\end{equation}
where $Q_{z_1,z}$ denotes the entry locus of $Y$ with respect to a general point on $\langle z_1, z \rangle$.
Similarly  
\begin{equation}\label{eq: 3, step birationality}
  \langle T_{z_2}(Y),z_1, z \rangle \cap Y \setminus \langle T_{z_2}(Y),z_1 \rangle 
= \pi_{z_2,z_1}^{-1}(\pi_{z_2,z_1}(z))=Q_{z_2,z}\setminus \langle T_{z_2}(Y),z_1 \rangle.
\end{equation}
Now 
$
 \pi_{L_q}^{-1}(\pi_{L_q}(z))=\langle L_q,z \rangle\cap Y \setminus \Sigma_q 
$
and, by the generality of $z$, 
\begin{equation}\label{eq: 4, step birationality}
\pi_{L_q}^{-1}(\pi_{L_q}(z))=\langle L_q,z \rangle\cap Y \setminus H_q.
\end{equation}
By (\ref{eq: 1, step birationality}), (\ref{eq: 2, step birationality}), (\ref{eq: 3, step birationality}) and (\ref{eq: 4, step birationality}) 
and observing that the spaces $\langle T_{z_1}(Y),z_2\rangle$, 
$\langle T_{z_2}(Y),z_1\rangle$ are contained in $H_q$,
it follows 
\begin{equation}\label{eq: 5, step birationality}
\{z\}\subseteq \pi_{L_q}^{-1}(\pi_{L_q}(z))\subseteq Q_{z_1,z}\cap Q_{z_2,z}.
\end{equation}
Finally, as
 we have already observed  in  Claim \ref{step: gammaY=deltaY}, the restriction
of the tangential projection
 $\tau_{z_1,Y}$ to $Q_{z_{2},z}$ is an isomorphism
$\bar{\tau}:=(\tau_{z_1,Y})|_{Q_{z_{2},z}}:Q_{z_2,z}\rightarrow W_{z_1,Y}$;
hence 
\begin{equation}\label{eq: 6, step birationality}
\{z\}=\bar{\tau}^{-1}(\bar{\tau}(z))=\tau_{z_1,Y}^{-1}(\tau_{z_1,Y}(z))\cap Q_{z_2,z}=Q_{z_1,z}\cap Q_{z_2,z}. 
\end{equation}
By (\ref{eq: 5, step birationality}) and (\ref{eq: 6, step birationality}), it follows $\pi_{L_q}^{-1}(\pi_{L_q}(z))=\{z\}$
and hence the birationality of $\pi_{L_q}$.
\end{proof}
\begin{step}\label{step: isomorphism}
$\pi_{L_q}$ induces an isomorphism
$Y\setminus H_q\stackrel{\simeq}{\longrightarrow} \PP^{16}\setminus H'_q$.
\end{step}
\begin{proof}[of the Claim]
We resolve  the indeterminacies of $\pi_{L_q}$ with the diagram
\begin{displaymath}
\xymatrix{& \Bl_{\Sigma_{q}}(Y) \ar[dl]_{\alpha} \ar[dr]^{\widetilde{\pi_{L_q}}}\\
Y\ar@{-->}[rr]^{\pi_{L_q}}& & \PP^{16}}
\end{displaymath}
The morphism $\widetilde{\pi_{L_q}}$ is projective and birational
 and hence surjective. Moreover, the
  points of the base locus  of $\pi_{L_q}^{-1}$ are the points for which
  the fiber of $\widetilde{\pi_{L_q}}$ has positive dimension.
Since $H_q\supseteq L_q$ and $H_q\cap Y=\overline{\pi_{L_q}^{-1}(H'_q)}$,
in order to prove the assertion, it suffices to
show  that, for every $w\in\PP^{16}\setminus H'_q$,
 $\dim\left(\widetilde{\pi_{L_q}}^{-1}(w)\right)=0$. 
Suppose by contradiction that there exists $w\in \PP^{16}\setminus H'_q$ 
such that $Z:=\widetilde{\pi_{L_q}}^{-1}(w)$ has positive dimension. 
Then, for the choice  of $H_q$, 
we have 
$\emptyset=Z\cap \alpha^{-1}(H_q\cap Y)\supseteq Z\cap \alpha^{-1}(\Sigma_{q})$ and 
therefore 
$\alpha(Z)$ contains an irreducible curve $C$ 
with $\pi_{L_q}(C)=w$ and $C\cap L_q=\emptyset$, against the fact
that a linear projection, when is defined everywhere, is a finite morphism. 
\end{proof}
\begin{step}\label{step: Y smooth}
$Y$ is smooth.
\end{step}
\begin{proof}[of the Claim]
Suppose that there exists a point
 $z_0$ with
$$z_0\in
\bigcap_{ q\in\Sec(Y) \atop \mathrm{generale} } T_q\left(\Sec(Y)\right)\cap Y = 
  \bigcap_{q\in\Sec(Y)} T_q\left(\Sec(Y)\right)\cap Y =
  \mathrm{Vert}\left(\Sec(Y)\right)\cap Y .$$
If $z\in Y$ is a general point, since $Y$ is not a cone, 
the tangential projection $\tau_{z,Y}$ is defined in $z_0$ and
it follows that 
$\tau_{z,Y}(z_0)$ is a vertex of $W_{z,Y}$.
This contradicts Claim \ref{step: gammaY=deltaY} and
 hence we have
$\bigcap_{ q\in\Sec(Y) \atop \mathrm{generale} } T_q\left(\Sec(Y)\right)\cap Y=\emptyset$,
from which we conclude by Claim \ref{step: isomorphism}.
\end{proof}
Now we can conclude the proof of
 Theorem \ref{prop: classification of type 2-2 into quadric}.
 By  Claim \ref{step: Y smooth} it follows that $Y\subset\PP^{26}$ 
is a Severi variety, so by their 
classification (see Table \ref{tab: severi varieties} or 
directly \cite[\Rmnum{4} Theorem~4.7]{zak-tangent}) it follows
 that $Y=E_6$;
moreover, since
$
27=h^0(\PP^{26}, \I_{\B',\PP^{26}}(2))\leq 
h^0(\PP^{26}, \I_{Y,\PP^{26}}(2))=27,
$
we have $Y=\B'$. Now, by the
classification of the special Cremona transformations of type $(2,2)$
 in \cite[Theorem~2.6]{ein-shepherdbarron}
(or also by a direct calculation), it follows that the lifting 
$\psi:\PP^{26}\dashrightarrow\PP^{26}$ 
of $\varphi^{-1}:\Q\dashrightarrow\PP^{25}=:H\subset\PP^{26}$ is
a birational transformation of type $(2,2)$ and therefore 
the base locus  $\widehat{\B}\subset\PP^{26}$ of the inverse of 
 $\psi$ is again the variety $E_6$.
Of course $\widehat{\B}\cap H = \B$ and hence the thesis.
\end{proof}
\begin{remark}
We observe that from
 Claim \ref{step: gammaY=deltaY} and 
\cite[Proposition~5.13]{chiantini-ciliberto}
it follows
that $Y$ is a \emph{$R_1$-variety} and hence by
\cite[Theorem~8.7]{chiantini-ciliberto}
it follows
Claim \ref{step: Y smooth}.
\end{remark}
\begin{remark} 
Note that the proof of Theorem \ref{prop: classification of type 2-2 into quadric} 
could  be  simplified
if one had shown  a priori  that
the map $\varphi:\PP^{n}\dashrightarrow\PP^{n+1}$ is the restriction of  a birational map
$\widehat{\varphi}:\PP^{n+1}\dashrightarrow\PP^{n+1}$  of type $(2,2)$.
In fact, if it were the  case, denoting by
$\widehat{\B}$ the base locus 
 of $\widehat{\varphi}$,  one could deduce that:
\begin{itemize}
\item  $\sing(\widehat{\B})$ is a finite set 
(otherwise  the very ample divisor 
 $\B$ of $\widehat{\B}$ would intersect an irreducible curve contained in 
 $\sing({\widehat{\B}})$ and then $\B$ would be singular, see \cite{debarre});
\item  $\widehat{\B}$ cannot be a cone 
(this follows by the birationality of $\widehat{\varphi}$).
\end{itemize}
Thus, restricting $\widehat{\varphi}$ to a general hyperplane,
we get a special quadratic birational transformation into a general quadric containing $\B'$,  
for which  $\B$ would be  a general  hyperplane section of
 $\widehat{\B}$; in particular 
$\widetilde{\gamma}(\B)\geq\widetilde{\gamma}(\widehat{\B})\geq 0$.
Now, as in Claim \ref{step: smooth image tangential projection}, 
one would deduce that $\widetilde{\gamma}(\B)=0$ and
hence that $\widehat{\B}$ is a Severi variety, by \cite{chiantini-ciliberto}.
\end{remark}

\section{Transformations whose base locus has dimension \texorpdfstring{$\leq3$}{<=3}}\label{sec: small dimension of B}
Let $\varphi$ be a special transformation  as
in Notation \ref{notation: factorial hypersurface} and let $r\leq3$.  
From  Proposition \ref{prop: dimension formula} 
we get the following possibilities for $(r,n)$:    $(1,4)$; $(2,6)$;  $(3,7)$; $(3,8)$.
If $(r,n)\in\{(1,4),(3,7)\}$ then $(d,\Delta)=(2,2)$
and these cases  
have already been classified in 
Theorem \ref{prop: classification of type 2-2 into quadric}.
\subsection{Case \texorpdfstring{$(r,n)=(2,6)$}{(r,n)=(2,6)}}
Lemma \ref{prop: castelnuovo argument} is  Castelnuovo's classic argument.
\begin{lemma}\label{prop: castelnuovo argument}  
 If $\Lambda\subset\PP^{c}$ is a set 
 of $\lambda\leq 2c+1$ points in general position, then
$\Lambda$ imposes independent conditions to the quadrics of $\PP^c$, i.e. 
$
h^0(\PP^{c},\I_{\Lambda,\PP^c}(2))=
\left( c+1\right) \left( c+2\right) /2-\lambda
$.
\end{lemma}
\begin{proof}
We can assume $\lambda=2c+1$.     
Put $\Lambda=\{p_0,\ldots,p_{2c}\}$ 
and consider the hyperplanes  $H_1=\langle p_1,\ldots,p_c\rangle$ and
 $H_2=\langle p_{c+1},\ldots,p_{2c}\rangle$.     
Since the points are in general position,
the quadric  
 $H_1\cup H_2$ contains the points  $p_1,\ldots,p_{2c}$,  but not  $p_0$.     
This proves the exactness of the sequence 
$0\rightarrow H^0(\PP^c,\I_{\Lambda,\PP^c}(2))\rightarrow 
H^0(\PP^c, \O_{\PP^c}(2))\rightarrow
 H^0(\Lambda, \O_{\Lambda}(2))=\bigoplus_{i=0}^{2c}\CC\rightarrow 0 $,
from which the assertion follows.
\end{proof}
\begin{proposition}\label{prop: 2-fold in P6}
 Let $\varphi:\PP^6\dashrightarrow\overline{\varphi(\PP^6)}=\sS\subset\PP^7$ be
 birational and special of type $(2,d)$,   with
$\sS$ a factorial hypersurface of degree $\Delta\geq2$.  
If $r=\dim(\B)=2$, then
 $\B$ is the blow-up $\sigma:\Bl_{\{p_0,\ldots,p_5\}}(\PP^2)\rightarrow\PP^2$ 
of $6$ points in the plane with $H_{\B}\sim \sigma^{\ast}(4H_{\PP^2})-2E_0-E_1-\cdots-E_5$ ($E_0,\ldots,E_5$ are the exceptional divisors).     
Moreover, we have $d=3$ and $\Delta=2$.
\end{proposition}
\begin{proof}
 By Lemma \ref{prop: cohomology twisted ideal}, it follows 
$\chi(\B,\O_{\B}(1))=7$ and $\chi(\B,\O_{\B}(2))=20$,   from which we deduce
$$
P_{\B}(t)=\left(\lambda\,{t}^{2}+\left( 26-3\,\lambda\right) \,t+2\,\lambda-12\right)/2=\left(\left( g+12\right)\,{t}^{2} +\left( 16-3\,g\right)\,t+2\,g \right)/4
$$
and hence 
$g=2(\lambda-6)$.     In 
particular $\lambda\geq6$,   being $g\geq0$.
Now, if $\lambda\leq 2\,\mathrm{codim}_{\PP^6}(\B)+1=9$,   cutting 
 $\B$ with a general $\PP^4\subset\PP^6$,
we obtain a set 
$\Lambda\subset\PP^4$ of $\lambda$ points
 that imposes independent conditions to the quadrics
 of $\PP^4$ (Lemma \ref{prop: castelnuovo argument});
hence 
$
 h^0(\PP^6,\I_{\B,\PP^6}(2))\leq h^0(\PP^4,\I_{\Lambda,\PP^4}(2))=h^0(\PP^4,\O_{\PP^4}(2))-\lambda 
$, 
 i.e.  $\lambda\leq7$.   
  Moreover, if $\lambda\geq 9$,   $\Lambda$ would impose at 
least 
 $9$ conditions to the quadrics  and hence we would get the contradiction 
$h^0(\PP^6,\I_{\B
}(2))\leq 6$.
Hence $\lambda=6$ or $\lambda=7$,   and in both cases, knowing 
the expression of the Hilbert polynomial, we conclude
applying 
 \cite{ionescu-smallinvariants}:
if $\lambda=6$,   such a variety does not exist; 
if $\lambda=7$,   then $\B$ is as asserted. 
Finally,  by
Remark \ref{remark: dimension formula without hypothesis}, 
we get either $(d,\Delta)=(3,2)$ or $(d,\Delta)=(2,3)$, 
but the latter case  is impossible by Example \ref{example: d=3 Delta=2}
(the pair $(d,\Delta)$ may also be determined 
by calculating the Chern classes of $\B$, as in Remark \ref{rem: chern classes}). 
\end{proof}

\subsection{Case \texorpdfstring{$(r,n)=(3,8)$}{(r,n)=(3,8)}}
Firstly we observe that if $(r,n)=(3,8)$  by Remark \ref{remark: dimension formula without hypothesis} 
it follows  $d+\Delta=6$ and hence we have $(d,\Delta)\in\{(2,4),(3,3),(4,2)\}$.
\begin{remark}\label{rem: chern classes}
See also \cite{crauder-katz-1989} and \cite{crauder-katz-1991}.
Let notation be as in the proof of 
Proposition \ref{prop: dimension formula} and  let $(r,n)=(3,8)$. 
Denote by $c_j:=c_j(\mathcal{T}_{\B})\cdot H_{\B}^{3-j}$ 
(resp. $s_j:=s_j(\mathcal{N}_{\B,\PP^8})\cdot H_{\B}^{3-j}$), for $1\leq j\leq 3$,
the degree of the $j$-th Chern class (resp. Segre class) of $\B$.
From the exact sequence
$0\rightarrow\mathcal{T}_{\B}\rightarrow \mathcal{T}_{\PP^8}|_{\B}\rightarrow\mathcal{N}_{\B,\PP^8}\rightarrow0$
we get: $s_1=c_1-9\lambda$,
$s_2=c_2-9c_1+45\lambda$, 
$s_3=c_3-9c_2+45c_1-165\lambda$.
Moreover
\begin{eqnarray*}
\lambda&=& H_{\B}^3=-K_{\B}\cdot H_{\B}^2+2g-2-\lambda 
= s_1+8\lambda+2g-2, \\
d \Delta &=& d{H'}^{8}={H'}^{7}\cdot(dH'-E')=(2H-E)^7\cdot H \\
&=& -H\cdot E^7+14H^2\cdot E^6-84H^3\cdot E^5 + 128H^8 
= -s_2 -14 s_1 -84 \lambda +128, \\
\Delta &=& {H'}^8=(2H-E)^8  
= E^8-16H\cdot E^7+112H^2\cdot E^6-448H^3\cdot E^5+256H^8 \\
&=& -s_3-16s_2-112s_1 -448\lambda+256,
\end{eqnarray*}
and hence
\begin{displaymath}
\left\{
 \begin{array}{l} 
  s_1=-7\lambda-2g+2, \\
s_2=14\lambda+28g-d\Delta+100, \\
s_3=112\lambda-224g+(16d-1)\Delta-1568, 
 \end{array}
\right.
\left\{
 \begin{array}{l} 
c_1=2\lambda-2g+2, \\
c_2=-13\lambda+10g-d\Delta+118, \\
c_3=70\lambda-44g+(7d-1)\Delta-596. 
 \end{array}
\right.
\end{displaymath}
Also, if $S$ is a general hyperplane section of $\B$, 
from the exact sequence
 $0\rightarrow\mathcal{T}_{S}\rightarrow\mathcal{T}_{\B}|_{S}\rightarrow\O_{S}(1)\rightarrow0$,
 we deduce 
$c_2=c_2(S)+c_1(S)=12\chi(\O_{S})-K_{S}^2-K_{S}\cdot H_{S}$
and hence
\begin{displaymath}
K_S^2=14\lambda+12\chi(\O_S)-12g+d\Delta-116.
\end{displaymath}
\end{remark}

\begin{proposition}\label{prop: 3-fold in P8 - S nonlinear}
Let $\varphi:\PP^8\dashrightarrow\overline{\varphi(\PP^8)}=\sS\subset\PP^9$ be
 birational and special of type $(2,d)$, 
with $\sS$ a factorial hypersurface of degree $\Delta\geq2$.
 If $r=\dim(\B)=3$,
 then  one of the following cases holds:
\begin{enumerate}[(i)]
 \item\label{part: case 0, 3-fold} $\lambda=12$, $g=7$, $d=4$, $\Delta=2$, $\B$ is a linear section of the spinorial variety $S^{10}\subset\PP^{15}$;
 \item\label{part: case 0.5, 3-fold} $\lambda=12$, $g=7$, $d=2$, $\Delta=4$, $\B$ is a Mukai variety with Betti numbers  $b_2=2$, $b_3=18$;
 \item\label{part: case 1, 3-fold} $\lambda=11$, $g=5$, $d=3$, $\Delta=3$, $\B$ is the variety $\mathfrak{Q}_{p_1\ldots,p_5}$ defined as the blow-up 
of  $5$ points $p_1,\ldots,p_5$ (possibly infinitely near) in a smooth quadric $Q\subset\PP^4$,   
  with $H_{\mathfrak{Q}_{p_1,\ldots,p_5}}\sim \sigma^{\ast}(2{H_{\PP^4}}|_{Q})-E_1-\cdots-E_5$, 
where $\sigma$ is the blow-up map and $E_1,\ldots,E_5$ are 
the exceptional divisors; 
 \item\label{part: case 2, 3-fold} $\lambda=11$, $g=5$, $d=4$, $\Delta=2$, $\B$ is a scroll over 
                                    $\PP_{\PP^1}(\O\oplus\O(-1))$.  
\end{enumerate}
\end{proposition}
\begin{proof}
 By Proposition \ref{prop: cohomology properties},  
 $\B\subset\PP^8$ is nondegenerate and linearly normal. 
 Let $\Lambda\subset C\subset S \subset \B \subset \PP^8$ be a sequence of
 general linear sections 
 of $\B$.
By the exact sequence 
$ 0\rightarrow \I_{\B,\PP^{8}}(-1)\rightarrow\I_{\B,\PP^{8}}\rightarrow \I_{S, \PP^7} 
\rightarrow 0 $
and those similar for
 $S$ and $C$,  using that $\B$,   $S$,   $C$ are nondegenerate, we get
the inequality
$h^0(\PP^8,\I_{\B,\PP^8}(2))\leq h^0(\PP^5,\I_{\Lambda,\PP^5}(2))$.
In particular, putting
$h_{\Lambda}(2):=\dim(\mathrm{Im}(H^0(\PP^5,\O_{\PP^5}(2))\rightarrow 
H^0(\Lambda,\O_{\Lambda}(2))))$,
we have
\begin{equation}\label{eq: leq_h2}
h_{\Lambda}(2)\leq h^0(\PP^5,\O_{\PP^5}(2))-h^0(\PP^8,\I_{\B,\PP^8}(2))=11.
\end{equation}
If now $\#(\Lambda)=\lambda\geq 11$,   
taking $\Lambda'\subseteq\Lambda$ with  $\#(\Lambda')=11$,   
by Lemma \ref{prop: castelnuovo argument} we obtain 
\begin{equation}\label{eq: geq_h2}
h_{\Lambda}(2)\geq h^0(\PP^5,\O_{\PP^5}(2))-h^0(\PP^5,\I_{\Lambda',\PP^5}(2))=\#(\Lambda')=11.
\end{equation}
The inequalities
 (\ref{eq: leq_h2}) and (\ref{eq: geq_h2}) yield
$ h_{\Lambda}(2)=2\cdot6-1 $
and this, by \cite[Lemma~1.10]{ciliberto-hilbertfunctions}, yields a contradiction if $\lambda\geq 13$.
Thus we have 
$ \lambda\leq 12$
 and, by  Castelnuovo's bound \cite[page~252]{griffiths-harris}, 
we also have 
\begin{equation}\label{eq: K_S.H_Sleq0}
K_S\cdot H_S = (K_{\B}+H_{\B})\cdot H_{\B}^2=2g-2-\lambda\leq0.
\end{equation}
We discuss two cases.
\begin{case}[Suppose $K_S\nsim 0$] 
By (\ref{eq: K_S.H_Sleq0}) and 
by the proof of \cite[\Rmnum{5} Lemma~1.7]{hartshorne-ag},
it follows  that
$ h^2(S,\O_S)=h^2(S,\O_S(1))=0 $.
Consequently, by Lemma \ref{prop: cohomology twisted ideal} and by
the exact sequence
$ 0\rightarrow\O_{\B}(-1)\rightarrow\O_{\B}\rightarrow\O_S\rightarrow0 $,
 we obtain
$$h^2(\B,\O_{\B})=h^3({\B},\O_{\B})=h^3({\B},\O_{\B}(-1))=0.$$
Moreover $h^1(\B,\O_{\B})=h^1(S,\O_S)=:q$ and using again  
 Lemma \ref{prop: cohomology twisted ideal} we obtain
\begin{equation}\label{eq: conditions hilbert pol}
\begin{array}{cc}
\chi({\B},\O_{\B}(-1))=0, & \chi({\B},\O_{\B})=1-q,\\
\chi({\B},\O_{\B}(1))=9, &\chi({\B},\O_{\B}(2))=35.
\end{array} \end{equation}
Now, the conditions (\ref{eq: conditions hilbert pol}) determine
 $P_{\B}(t)$
in function of $q$, 
 from which in particular we obtain 
 $\lambda=11-3q$,  $g=5-5q$.
Being $g\geq0$,   we have $(q,\lambda,g)=(0,11,5)$ or $(q,\lambda,g)=(1,8,0)$, 
but the latter case
is impossible by \cite[Theorems~10.2 and 12.1, Remark~12.2]{fujita-polarizedvarieties}.
 Thus we have
\begin{equation}\label{eq: irregularity is 0}
q=0,\ P_{\B}(t)=\left( 11{t}^{3}+21{t}^{2}+16t+6 \right)/6,\ K_S\cdot H_S=-3,\ g=5.
\end{equation}
Applying the main result in \cite{ionescu-degsmallrespectcodim} and the numerical constraints 
 in \cite{besana-biancofiore-deg11} and \cite{besana-biancofiore-numerical}, it follows immediately that
 $\B$ is one of the following:
\begin{enumerate}[(a)] 
 \item\label{part: case 1, 3-fold - proof} the variety $\mathfrak{Q}_{p_1,\ldots,p_5}$; 
 \item\label{part: case 2, 3-fold - proof} a scroll over a surface $Y$, where  $Y$ is either the 
blow-up  of $5$ points in $\PP^2$, or the rational ruled surface $\PP_{\PP^1}(\O\oplus\O(-1))$;
 \item\label{part: case 3, 3-fold - proof} a quadric fibration over $\PP^1$.     
\end{enumerate}
Now, if $\B$ is as in case (\ref{part: case 2, 3-fold - proof}),
we use the well-known 
relation (multiplicativity of the topological Euler characteristic)
 $c_3(\B)=c_1(\PP^1)c_2(Y)$, and by Remark \ref{rem: chern classes} we deduce  
$\Delta=(2c_2(Y)+46)/(7d-1)$.
Moreover, if $Y$ is the blow-up of $5$ points in $\PP^2$, we have 
$c_2(Y)=12\chi(\O_Y)-K_Y^2=12\chi(\O_{\PP^2})-(K_{\PP^2}^2-5)=8$,
while if $Y$ is $\PP_{\PP^1}(\O\oplus\O(-1))$, we have $c_2(Y)=4$.
Thus, if $\B$ is as in case (\ref{part: case 2, 3-fold - proof}), 
we have $Y=\PP_{\PP^1}(\O\oplus\O(-1))$, $d=4$ and $\Delta=2$.
If $\B$ is as in case (\ref{part: case 3, 3-fold - proof}), 
we easily deduce that $c_2(\B)=20$ and hence, again by Remark \ref{rem: chern classes},
we obtain the contradiction $d\Delta=5$.
Now suppose $\B$ as in case (\ref{part: case 1, 3-fold - proof}), namely
 $\B$ is obtained as a sequence 
$$
\B=Z_5\stackrel{\sigma_5}{\longrightarrow}Z_4\stackrel{\sigma_4}{\longrightarrow}\cdots 
\stackrel{\sigma_1}{\longrightarrow} Z_0=Q,
$$
 where $\sigma_j$ is the blow-up at a point $p_j\in Z_{j-1}$, 
$H_{Z_{j}}=\sigma_j^{\ast}(H_{Z_{j-1}})-E_j$, $E_j$ is the exceptional divisor
 and $H_{Z_0}=H_Q=2H_{\PP^4}|_{Q}$.
By \cite[page~609]{griffiths-harris} it follows that 
$c_2(Z_j)=\sigma_j^{\ast}(c_2(Z_{j-1}))$ and hence 
$$ c_2(Z_j)\cdot H_{Z_j}=
 \sigma_j^{\ast}(c_2(Z_{j-1}))\cdot \sigma_j^{\ast}(H_{Z_{j-1}})-\sigma_j^{\ast}(c_2(Z_{j-1}))\cdot E_j=
 c_2(Z_{j-1})\cdot H_{Z_{j-1}}.$$
In particular,  we obtain 
$c_2(\B)\cdot H_{\B}=2c_2(Q)\cdot H_{\PP^4}|_Q=16$. 
On the other hand, by Remark \ref{rem: chern classes}, we obtain 
that $c_2(\B)\cdot H_{\B}=25-d\Delta$, hence $d\Delta=9$.
 \end{case}
\begin{case}[Suppose $K_S\sim 0$] 
By Castelnuovo's bound, since $K_S\cdot H_S=0$, it follows that 
$(\lambda,g)=(12,7)$ and hence also that $\chi(S,\O_{S})=-3\lambda+2g+24=2$ (note that,
 as in the previous case, we know the values of $P_{\B}(1)$ and $P_{\B}(2)$).
We have $q=h^1(S,\O_S)=1-\chi(S,\O_S)+h^2(S,\O_S)=-1+h^2(S,K_S)=0$ and hence 
$S$ is a $K3$-surface, $C$ is a canonical curve
  and $\B$ is a Mukai variety.
Now we denote by $b_j=b_j(\B)$ the $j$-th Betti number of $\B$.
By Poincar\'e-Hopf index formula and Poincar\'e duality 
(see for example \cite{griffiths-harris}) 
we have $c_3(\B)=\sum_{j} (-1)^j b_j= 2+2b_2-b_3$ 
and, by Remark \ref{rem: chern classes}, we also have 
$c_3(\B)=-7d^2+43d-70$. Moreover, by \cite{mori-mukai}, if $b_2\geq 2$ then 
$(b_2,b_3)\in\{(2,12),(2,18),(3,16),(9,0)\}$. Thus, 
if $b_2\geq 2$ we have $b_2=2$, $b_3=18$, $d=2$, $\Delta=4$.
Finally, by \cite{mukai-biregularclassification},
 if $b_2=1$ then $\B$ is a linear section 
of the spinorial variety $S^{10}\subset\PP^{15}$. 
Thus, we have a natural inclusion 
$\iota:H^0(\PP^{15},\I_{S^{10}}(2))\hookrightarrow H^0(\PP^{8},\I_{\B}(2))$
and, since $h^0(\PP^{15},\I_{S^{10}}(2))=h^0(\PP^{8},\I_{\B}(2))$,
we see that $\iota$ is an isomorphism.
This says that $\varphi$ is the restriction of the map 
$\psi:\PP^{15}\dashrightarrow\Q\subset\PP^9$ given in Example \ref{example: d=4 Delta=2}.
\end{case} 
\end{proof}

\begin{remark}[on case (\ref{part: case 2, 3-fold}) 
               of Proposition \ref{prop: 3-fold in P8 - S nonlinear}]
More precisely, from \cite[Proposition~4.2.3]{besana-biancofiore-deg11} it follows 
that $\B=\PP_{\mathbb{F}_1}(\mathcal{E})$,
where $(\mathbb{F}_1,H_{\mathbb{F}_1}):=(\PP_{\PP^1}(\O\oplus\O(-1)),C_0+2f)$ 
(notation as in \cite[page~373]{hartshorne-ag}) 
and $\mathcal{E}$ is a locally free sheaf of rank $2$ 
on $\mathbb{F}_1$, with $c_2(\mathcal{E})=10$.
$\mathbb{F}_1$ is thus the cubic surface of $\PP^4$ with ideal generated by: 
$x_0x_3-x_2x_1, x_0x_4-x_3x_1, x_2x_4-x_3^2$ and it is 
  isomorphic to
$\PP^2$ with one point blown up. 
We point out that
 the problem of the existence of an example for case (\ref{part: case 2, 3-fold}) 
of Proposition \ref{prop: 3-fold in P8 - S nonlinear}
  is essentially reduced to showing that 
such a scroll over $\FF_1$ must be cut out by quadrics. In fact,
by \cite{alzati-fania-ruled} (see also \cite{besana-fania-flamini-f1}) there exists
a smooth irreducible nondegenerate linearly normal $3$-dimensional variety 
 $X\subset\PP^8$  with
 $h^1(X,\O_X)=0$, 
 degree $\lambda=11$, sectional genus $g=5$, having the structure of a 
 scroll $\PP_{\FF^1}(\E)$  with $c_1(\E)=3C_0+5f$ and $c_2(\E)=10$ 
and hence having degrees of the Segre classes 
 $s_1(X)=-85$, $s_2(X)=386$, $s_3(X)=-1330$.
Now, by \cite[Proposition~2]{alzati-russo-subhomaloidal}, $X\subset\PP^8$ is 
arithmetically Cohen-Macaulay
and by Riemann-Roch, denoting with $C$ a general curve section of $X$,
 we obtain
\begin{eqnarray*} 
h^0(\PP^8,\I_X(2)) &=& h^0(\PP^6,\I_C(2))
= h^0(\PP^6,\O_{\PP^6}(2))-h^0(C,\O_C(2)) \\
&=& 28-(2\lambda+1-g) = 10.
\end{eqnarray*}
If the homogeneous ideal of $X$ is generated by quadratic forms or at least if
$X=V(H^0(\I_X(2)))$,
the linear system $|\I_X(2)|$ defines a rational map 
$\psi:\PP^8\dashrightarrow\sS=\overline{\psi(\PP^8)}\subset\PP^{9}$ 
whose base locus is
 $X$ and whose image $\sS$ is nondegenerate.
Now, denoting  with $\pi:\Bl_X(\PP^8)\rightarrow \PP^8$ 
the blow-up of $\PP^8$ along $X$, as in Remark \ref{rem: chern classes}, 
we deduce
\begin{eqnarray*} 
\deg(\psi)\deg(\sS) & = & (2\pi^{\ast}(H_{\PP^8})-E_X)^8 \\ &=&
-s_3(X)-16s_2(X)-112s_1(X)-448\deg(X)+256 = 2,
\end{eqnarray*}
from which 
$\deg(\psi)=1$ and $\deg(\sS)=2$.
\end{remark}
\begin{proposition}\label{prop: 3-fold in P8}
 Let $\varphi:\PP^8\dashrightarrow\overline{\varphi(\PP^8)}=\sS\subset\PP^9$ be
 birational and special of type $(2,d)$, with   
  $\sS$ a hypersurface as in Notation \ref{notation: factorial hypersurface}.
If $r=\dim(\B)=3$,
then either 
case (\ref{part: case 0, 3-fold}), 
case (\ref{part: case 1, 3-fold}), or 
case (\ref{part: case 2, 3-fold}) of 
Proposition \ref{prop: 3-fold in P8 - S nonlinear} 
holds.
\end{proposition} 
\begin{proof}
 We have to exclude case (\ref{part: case 0.5, 3-fold}) 
of proposition \ref{prop: 3-fold in P8 - S nonlinear}, so
we just assume that $\B$ is as in this case.
Using the fact that $K_{\B}\sim -H_{\B}$ and $(\lambda,g)=(12,7)$ 
we can compute
the Segre classes of the tangent bundle of $\B$:
\begin{eqnarray*}
 s_1(\mathcal{T}_{\B})\cdot H_{\B}^2 &=& -\lambda =-12 , \\
 s_2(\mathcal{T}_{\B})\cdot H_{\B}   &=&  -24+\lambda=-12 , \\
 s_3(\mathcal{T}_{\B})               &=&  -c_3(\B)+48-\lambda=100-(7d-1)\Delta. \\
\end{eqnarray*}
Since $\B$ is a $QEL$-variety of type $\delta=0$ we apply the
\emph{double point formula} 
(see for example \cite{peters-simonis} and \cite{laksov})
\begin{equation}
 2(2d-1) = 2\deg(\Sec(\B)) 
= \lambda^2 - \sum_{j=0}^{3}{7 \choose j } s_{3-j}(\mathcal{T}_{\B})\cdot H_{\B}^{j} 
= (7d-1)\Delta-40,
\end{equation}
from which we deduce 
$\Delta=(4d+38)/(7d-1)$ i.e. $d=4$ and $\Delta=2$.
\end{proof}
\begin{proof}[Second proof of Proposition \ref{prop: 3-fold in P8}]
Let $x\in\B$ be a general point and put $k=\#(\L_{x,\B})$.
Since the variety $\B$ is not a scroll over a curve  
(otherwise it would happen $\lambda^2\geq (2r+1)\lambda+r(r+1)(g-1)$, by \cite{besana-biancofiore-numerical}) 
and it is defined by
 quadrics, by \cite[Proposition~5.2]{ciliberto-mella-russo}, 
it follows that 
the support of the base locus
 of the tangential projection $\tau_{x,\B}:\B\dashrightarrow W_{x,\B}\subset\PP^4$, 
  i.e. $(T_x(\B)\cap\B)_{\mathrm{red}}$,   consists of $0\leq k< \infty$ lines through  $x$.
Now, by Proposition \ref{prop: dimension formula}, 
$\B$ is a $QEL$-variety of type $\delta=0$ and 
repeating the argument  
in \cite[\S  5]{ciliberto-mella-russo} (keeping also in mind 
 \cite[Theorem~2.3]{ionescu-russo-qel2}) we get  the relation 
$\lambda-8+k=\deg(W_{x,\B})$.
On the other hand, 
by proceeding as in Claim \ref{step: tang proj and entry locus}
or in \cite[Theorem~1.4]{mella-russo-baselocusleq3}, we also obtain
$\deg(W_{x,\B})\leq d$. Hence,  we deduce
\begin{equation}\label{eq: lambda-8+k leq d}
\lambda-8+k\leq d,
\end{equation}
from which the conclusion follows.
\end{proof}
\begin{remark}
Note that in case (\ref{part: case 1, 3-fold}) 
of Proposition \ref{prop: 3-fold in P8 - S nonlinear},
 by (\ref{eq: lambda-8+k leq d}) it follows that
$k=\#(\L_{x,\B})=0$.
We show directly that for a general point 
$x\in\mathfrak{Q}=\mathfrak{Q}_{p_1,\ldots,p_5}$,
we have $\L_{x,\mathfrak{Q}}=\emptyset$.
Suppose by contradiction that there exists $[l]\in\L_{x,\mathfrak{Q}}$.
Then 
$ 0=\dim_{[l]}(\L_{x,\mathfrak{Q}})=H^0(\PP^1,{\mathcal{N}}_{l,\mathfrak{Q}}(-1))=H^0(\PP^1,{{\mathcal{T}}_{\mathfrak{Q}}}|_l(-1))-2= -K_{\mathfrak{Q}}\cdot l -2$,
and hence
\begin{equation}\label{eq: (K+2H)l=0}
(K_{\mathfrak{Q}}+2H_{\mathfrak{Q}})\cdot l = 0.
\end{equation}
Moreover, by \cite[\S  0.3]{ionescu-degsmallrespectcodim}, the adjunction map
 $\psi_{\mathfrak{Q}}$, i.e. the map defined by the complete linear system
 $|K_{\mathfrak{Q}}+2\,H_{\mathfrak{Q}}|$,    is everywhere defined  and we have  
a commutative diagram of adjunction maps
$$
\xymatrix{
 (\mathfrak{Q},H_{\mathfrak{Q}})\ar[r]^{\psi_{\mathfrak{Q}}}  \ar[d]_{\sigma} & \psi_{\mathfrak{Q}}(\mathfrak{Q}) \\ (Q, H_{Q}) \ar[ur]_{\psi_{Q}}
}
$$
where $\sigma$ is the blow-up map and $(Q,H_Q)=(Q^3\subset\PP^4,2\,H_{\PP^4}|_{Q})$.
Now
$ K_{Q}+2\,H_{Q}  \sim (K_{\PP^4}+Q)|_{Q}+2\,(2\,H_{\PP^4})|_{Q}  
 \sim (-5\,H_{\PP^4}+2\,H_{\PP^4}+4\,H_{\PP^4})|_{Q}\sim H_{\PP^4}|_{Q} $,
but this is in contradiction with (\ref{eq: (K+2H)l=0}). 
\end{remark}

\begin{corollary}\label{prop: classification type 2-3 into cubic}
Let $\varphi$ be of type $(2,3)$ and let $\Delta=3$.
Then $\B$ is the variety $\mathfrak{Q}_{p_1,\ldots,p_5}\subset\PP^8$.
\end{corollary}
\begin{proof}
By Propositions \ref{prop: dimension formula} and \ref{Prop: numerical restrictions},
 $\B$ is a $QEL$-variety of type $\delta=0$ and dimension $3$ in $\PP^{8}$.
Hence we apply Proposition \ref{prop: 3-fold in P8}.
\end{proof}

\subsection{Summary table}
 Table \ref{tab: cases with dimension leq 3} 
classifies 
all special transformations 
as in Notation \ref{notation: factorial hypersurface} and with $r\leq 3$. 
\begin{table}[htbp]
\begin{center}
\begin{tabular}{|c|c|c|c|c|c|c|c|}
\hline
 $r$ & $n$ & $\Delta$ & $d$ & $\delta$ & $\lambda$  & Abstract structure of $\B$ & Examples\\
\hline
\hline
 $1$ & $4$ & $2$      &  $2$& $0$      & $4$       & $(\PP^1,\O(4))$ & exist\\
\hline
 $2$ & $6$ & $2$      &  $3$& $0$      & $7$       & Edge variety & exist\\
\hline
 $3$ & $7$ & $2$      &  $2$& $1$      & $6$       & Hyperplane section  of $\PP^2\times\PP^2\subset\PP^8$ & exist\\
\hline
 $3$ & $8$ & $2$      & $4$ & $0$     & $12$       & Linear section of $S^{10}\subset\PP^{15}$ & exist\\
\hline
$3$ & $8$ & $3$      & $3$ & $0$     & $11$        &  $\mathfrak{Q}_{p_1,\ldots,p_5}$ & exist \\
\hline
$3$ & $8$ & $2$      & $4$ & $0$     & $11$        &  Scroll over $\PP_{\PP^1}(\O\oplus \O(-1))$ & not know\\
\hline
\end{tabular}
\end{center}
\caption{Cases with $r\leq3$.}
\label{tab: cases with dimension leq 3}
\end{table}
We point out that in \cite{note2}
 we have extended this table
in the case in which $\sS$ is not necessarily a hypersurface.

\section{Invariants of transformations of
type \texorpdfstring{$(2,2)$}{(2,2)}
into a cubic and a quartic}\label{sec: invariants 2-2}

\begin{remark}\label{rem: LxB linearly normal}
Let $\delta\geq 3$ and consider $\L_{x,\B}\subset\PP^{r-1}$, 
where  $x\in\B$ is a general point. 
By \cite[Theorem~2.3]{russo-qel1} and
\cite[Corollary~1.6]{russo-linesonvarieties}
it follows that 
$\L_{x,\B}$ is a smooth irreducible nondegenerate
variety  
of codimension $(r-\delta+2)/2$ and it is scheme-theoretic intersection of
 quadrics.
Then, applying \cite[Corollary~2]{bertram-ein-lazarsfeld}, we get that $\L_{x,\B}$ is
linearly normal.
\end{remark}
In  Propositions \ref{prop: invariants d=2 Delta=3} and \ref{prop: invariants d=2 Delta=4} we write 
$P=a_0,a_1,\ldots, a_r$ to indicate that 
$
P_{\B}(t)=a_0 { t \choose r } + a_1 { t \choose r-1 } + \cdots + a_r 
$.
\begin{proposition}\label{prop: invariants d=2 Delta=3} Let $\varphi$ be of type $(2,2)$ 
and let $\Delta=3$. 
Then
$\B$ is a $QEL$-variety of type $\delta$ 
and a Fano variety of the first species 
of  coindex $c$,   as one of the following cases: 
\begin{enumerate}[(i)]
 \item\label{part: first, invariants d=2 Delta=3} $n=18$,   $r=10$,   $\delta=4$,   $c=4$, 
$P=$ $34$, $272$, $964$, $1988$, $2633$, $2330$, $1387$, $544$, $133$, $18$, $1$;
 for the general point $x\in\B$, $\L_{x,\B}\subset\PP^9$ is projectively equivalent to 
$\PP^1\times\PP^4\subset\PP^9$.
\item\label{part: second, invariants d=2 Delta=3} $n=24$,   $r=14$,   $\delta=6$,   $c=5$, 
$P=$ $80$, $920$, $4866$, $15673$, $34302$, $53884$, $62541$, $54366$, $35472$, $17228$, $6104$, $1521$, $250$, $24$, $1$; 
 for the general point $x\in\B$, $\L_{x,\B}\subset\PP^{13}$ is 
projectively equivalent to a smooth $8$-dimensional linear section  
  of $S^{10}\subset\PP^{15}$.
\end{enumerate}
\end{proposition}
\begin{proof}
By Proposition \ref{prop: dimension formula} and Proposition \ref{prop: cohomology properties} 
parts \ref{part: B is linearly normal} and \ref{part: B is fano}, 
and applying   \cite[Theorem~2.2]{ionescu-russo-conicconnected} and 
\cite[Theorem~2.8]{russo-qel1},
it follows that
 $\B$ is a $QEL$-variety of type $\delta$ and 
$$(n,r,\delta)\in\{(6,2,0),
(12,6,2), (18,10,4), (24,14,6)\}.$$
 The tern $(6,2,0)$ is excluded by Proposition  \ref{prop: 2-fold in P6}; the tern 
$(12,6,2)$ is excluded since otherwise 
by Proposition \ref{prop: cohomology properties} part \ref{part: hilbert polynomial},
we would get incompatible conditions 
for $P_{\B}(t)$. 
The statement on $\L_{x,\B}$, in the case (\ref{part: first, invariants d=2 Delta=3}) 
follows from
\cite[Theorem~2.2]{ionescu-russo-conicconnected}, 
while in the case (\ref{part: second, invariants d=2 Delta=3}) 
it follows  from \cite{mukai-biregularclassification} or \cite[Corollary~3.2]{russo-qel1}.
\end{proof}

\begin{proposition}\label{prop: invariants d=2 Delta=4}
Let $\varphi$ be of type $(2,2)$ and let $\Delta=4$. 
 Then $\B$ 
is a $QEL$-variety of type $\delta$ 
and a Fano variety 
of the first species
of coindex $c$, as one of the following cases:
\begin{enumerate}[(i)]
\item\label{part: first, invariants d=2 Delta=4} $n=17$,   $r=9$,   $\delta=3$,   $c=4$, 
$P=$ $35$, $245$, $747$, $1297$, $1406$, $980$, $435$, $117$, $17$, $1$;
 for the general point $x\in\B$, $\L_{x,\B}\subset\PP^{8}$ is 
projectively equivalent to 
$\PP_{\PP^1}(\O(1)\oplus\O(1)\oplus\O(1)\oplus\O(2))\subset\PP^8$.
\item\label{part: second, invariants d=2 Delta=4} $n=23$,   $r=13$,   $\delta=5$,   $c=5$, 
$P=$ $82$, $861$, $4126$, $11932$, $23195$, $31943$, $31984$, $23504$, $12628$, $4875$, $1306$, $228$, $23$, $1$; 
  for the general point  $x\in\B$, $\L_{x,\B}\subset\PP^{12}$ is 
projectively equivalent to a 
 smooth $7$-dimensional linear section  of $S^{10}\subset\PP^{15}$.
\end{enumerate}
\end{proposition}
\begin{proof}
 As in the proof  of Proposition \ref{prop: invariants d=2 Delta=3}, we get that $\B$ is a 
$QEL$-variety of type $\delta$ and
dimension $r$ with
$$(n,r,\delta)\in\{(8,3,0), (11,5,1), (17,9,3), (23,13,5)\}.$$
The case with $\delta=0$ is excluded by  Proposition \ref{prop: 3-fold in P8};
the case with $\delta=1$ is excluded by  Proposition 
\ref{prop: cohomology properties} part \ref{part: hilbert polynomial}; 
by the same Proposition, we get the expression of the Hilbert polynomials in the 
cases
 with $\delta\geq3$.     
Finally, the statement on $\L_{x,\B}$, in the case (\ref{part: first, invariants d=2 Delta=4}) 
follows from \cite[Theorem~2.2]{ionescu-russo-conicconnected}, 
while in the case (\ref{part: second, invariants d=2 Delta=4}) it follows  
from \cite{mukai-biregularclassification}
(for the latter case, by  Kodaira Vanishing Theorem  and  Serre Duality, 
we get $g(\L_{x,\B})=7$).   
\end{proof}

\section{Complements}\label{sec: nonspecial case}
In this section we treat the nonspecial case when
 $n\leq 4$. 
Precisely, we keep the following notation:
\begin{notation}\label{notation: nonspecial case}
 Let $\varphi:\PP^n\dashrightarrow\overline{\varphi(\PP^n)}=\Q\subset\PP^{n+1}$ be
a quadratic birational transformation into 
an irreducible (hence normal, 
by \cite[\Rmnum{1} Exercise~5.12, \Rmnum{2} Exercise~6.5]{hartshorne-ag}) quadric 
 hypersurface $\Q$ and
moreover suppose that  its base locus $\emptyset\neq\B\subset\PP^n$ is reduced. 
\end{notation}

\begin{lemma}\label{prop: degenerate component}
Let
$X\subseteq\B$ be a degenerate irreducible component 
 of $\B$.
\begin{enumerate}
\item If $\mathrm{codim}_{\PP^n}(X)=2$   then $\deg(X)\leq2$ and, 
if $\deg(X)=2$ then $X=\B$ and $\B$ is a quadric.
\item If $\mathrm{codim}_{\PP^n}(X)=3$   then $\deg(X)\leq4$ and, 
if $\deg(X)=4$ then we have $h^0(\PP^n,\I_{X}(2))=h^0(\PP^n,\I_{\B}(2))+1$.
\end{enumerate}
\end{lemma}
\begin{proof}
We can choose coordinates $x_0,\ldots,x_n$ on $\PP^n$ such that
$X\subset V(x_n)\subset\PP^n$ 
and we consider the 
restriction 
map
$u:H^0(\PP^n,\I_{\B}(2)) \rightarrow H^0(\PP^{n-1},\I_{X}(2))$,
defined by
$ 
u(F(x_0,\ldots,x_{n}))=F(x_0,\ldots,x_{n-1},0)$.
We remark that for every $F\in H^0(\PP^n,\I_{\B}(2))$ we have
$
F\in \langle u(F) \rangle \oplus \langle x_0x_n,\ldots,x_{n-1}x_{n},x_n^2 \rangle
$
and, in particular, $H^0(\PP^n,\I_{\B}(2))\subseteq 
\mathrm{Im}(u)\oplus \langle x_0x_n,\ldots,x_{n-1}x_{n},x_n^2 \rangle$.
Thus $\dim(\mathrm{Im}(u))\geq 1$ and, 
if $\dim(\mathrm{Im}(u))=1$ then   $\B$ is a quadric hypersurface in  $\PP^{n-1}$.     
Now suppose $\mathrm{codim}_{\PP^n}(X)=2$.     
From the above, there exists $\bar{F}\in H^0(\PP^{n-1},\I_{X}(2))$ and
$X$ has to be an irreducible component of $V(\bar{F})$.     
It follows that 
$\deg(X)\leq \deg(\bar{F})=2$ and, if $\deg(X)=2$ then  
$X=V(\bar{F})$ and 
$h^0(\PP^{n-1},\I_{X}(2))=\dim(\mathrm{Im}(u))= 1$.
 Suppose $\mathrm{codim}_{\PP^n}(X)=3$. 
 From the above, there exist  
 $\bar{F},\bar{F'}\in H^0(\PP^{n-1},\I_{X}(2))$ which are linearly independent
 and $X$ 
has to be contained in the complete intersection 
$V(\bar{F},\bar{F'})$.     
It follows that 
$\deg(X)\leq4$ and, if $\deg(X)=4$ then 
$h^0(\PP^{n-1},\I_{X}(2))= \dim(\mathrm{Im}(u))= 2$.
\end{proof}

\begin{proposition}\label{prop: P3}
If $n=3$,  
then either
\begin{enumerate}[(i)]
 \item\label{item: 1, P3} $\varphi$ is of type $(2,1)$,   or
 \item\label{item: 2, P3} $\varphi$ is of type $(2,2)$,   $\rk(\Q)=4$
 and $\B$ is the union of
a line  $r$ with two points $p_1,p_2$   such that
 $\langle p_1, p_2\rangle\cap r=\emptyset$.
\end{enumerate}
\end{proposition}
\begin{proof}
Since $\mathrm{codim}_{\PP^3}(\B)\geq2$,   $\B$ 
has the following decomposition into irreducible components:
$$
\B=\bigcup_{j}C_j\cup\bigcup_{i}p_i ,
$$ 
where $C_j$ and $p_i$ are respectively curves and points, with $\deg(C_j)\leq3$
by  B\'ezout's Theorem.
If one of $C_j$ is nondegenerate, 
then $\deg(C_j)=3$ and
 $C_j$ is the twisted cubic curve, by \cite[Proposition~18.9]{harris-firstcourse}. 
This would produce the absurd result that 
$5=h^0(\PP^3,\I_{\B,\PP^3}(2))\leq h^0(\PP^3,\I_{C_j,\PP^3}(2))=3$ and hence 
we have that every $C_j$ is degenerate.
By Lemma \ref{prop: degenerate component}, 
using the fact that a line
imposes $3$ conditions to the quadrics, 
it follows that 
 $\B$ is a (irreducible or not) conic
 or as asserted in  (\ref{item: 2, P3}).   
In the latter case, 
modulo a change of coordinates,
we can suppose
 $r=V(x_2,x_3)$,   $p_1=[0,0,1,0]$,   $p_2=[0,0,0,1]$ and  we get 
\begin{equation}
 \B=V(x_0x_2,\ x_0x_3,\ x_1x_2,\ x_1x_3,\ x_2x_3),\quad
 \Q= V(y_0y_3-y_1y_2).
\end{equation}
\end{proof}

\begin{lemma}\label{prop: isolated point}
 If $n\geq3$ and $\B$ has an isolated point, then  $\rk(\Q)\leq4$.
\end{lemma}
\begin{proof}
Let $p$ be an isolated (reduced) point of $\B$ and 
consider  the diagram
$$
\xymatrix{
\PP^{n-1}\simeq E \ar@{^{(}->}[r] \ar@{-->}@/^2.8pc/[rrd]^{\bar{\varphi}} &\Bl_{p}(\PP^{n}) \ar[d]^{\pi} \ar@{-->}[rd]\\
&\PP^{n} \ar@{-->}[r]^{\varphi} & \Q\subset\PP^{n+1}
}
$$
By the hypothesis on $p$, 
  $\bar{\varphi}$ is a linear morphism 
and so the quadric  $\Q=\Q^{n}$ contains a $\PP^{n-1}$.
This, since $n\geq3$,   implies $\rk(\Q)\leq 4$.
\end{proof}

\begin{proposition}\label{prop: P4}
If $n=4$ and
 $\rk(\Q)\geq 5$,   then either
\begin{enumerate}[(i)]
 \item\label{item: 1, P4} $\varphi$ is of type $(2,1)$,   or
 \item\label{item: 2, P4} $\varphi$ is of type $(2,2)$,   
$\Q$ is smooth and $\B$ 
is one of the following (Figure \ref{fig: fig_base_loci_n=4}):
\begin{enumerate} 
 \item the rational normal quartic curve, 
 \item the union of the twisted cubic curve  in a hyperplane  $H\subset\PP^4$ 
       with  a line not contained  in $H$ and intersecting the twisted curve,
 \item the union of an irreducible conic with two skew lines that intersect it,
 \item the union of three skew lines with another line that intersects them.
\end{enumerate}
\end{enumerate}
\end{proposition}
\begin{proof}
By Lemma \ref{prop: isolated point} and since $\mathrm{codim}_{\PP^4}(\B)\geq2$,   $\B$
has the following decomposition into irreducible components:
$$
\B=\bigcup_{i}S_i\cup\bigcup_{j}C_j ,
$$ 
where $S_i$ and $C_j$ are respectively surfaces and curves. 
Let $S$ be one of $S_i$ and let $C$ be one of $C_j$.
We discuss all possible cases.
\begin{case}[$\deg(S)\geq 4$] This case is impossible 
by    B\'ezout's Theorem.
\end{case}
 \begin{case}[$\deg(S)=3$,   $S$ nondegenerate] 
Cutting $S$ with a general hyperplane $\PP^3\subset\PP^4$, we obtain 
the twisted cubic curve $\Gamma\subset\PP^3$.
Hence, from the exact sequence 
$0\rightarrow \I_{S,\PP^{4}}(-1)\rightarrow
 \I_{S,\PP^{4}}\rightarrow \I_{\Gamma, \PP^3} \rightarrow 0$, 
we get the contradiction
$
6=h^0(\PP^{4},\I_{\B,\PP^{4}}(2))\leq h^0(\PP^{4},\I_{S,\PP^{4}}(2))\leq h^0(\PP^3,\I_{\Gamma, \PP^3}(2) )=3
$.
\end{case}
 \begin{case}[$\deg(S)\geq 2$,   $S$ degenerate] By 
Lemma \ref{prop: degenerate component} we have $\deg(S)=2$ and $S=\B$ is a quadric.
\end{case}
\begin{case}[$\deg(C)\geq 5$,   $C$ nondegenerate] Take a general 
hyperplane
 $\PP^3\subset\PP^4$
and put $\Lambda=\PP^3\cap C$. $\Lambda$ is a set of $\lambda\geq5$ points of $\PP^3$ 
in general position and therefore, by Lemma \ref{prop: castelnuovo argument}, 
we get the contradiction
$$
6=h^0(\PP^{4},\I_{\B}(2))\leq h^0(\PP^4,\I_{C}(2))\leq
h^0(\PP^3,\I_{\Lambda}(2))\leq
\left\{\begin{array}{ll}  10-\lambda\leq5, 
& \mbox{ if } \lambda\leq7 \\  10-7=3, & \mbox{ if } \lambda\geq7 . \end{array}   \right.
$$
\end{case}
 \begin{case}[$\deg(C)\geq5$,   $C$ degenerate]  
This is impossible by  Lemma \ref{prop: degenerate component}.
\end{case}
 \begin{case}[$\deg(C)=4$,   $C$ nondegenerate] By 
\cite[Proposition~18.9]{harris-firstcourse}, $C$ is the rational normal quartic curve  
and one of its parameterizations 
is $[s,t]\in\PP^1\mapsto [s^4,s^3t,s^2t^2,st^3,t^4]\in\PP^{4}$.
We have 
$h^0(\PP^4,\I_{C,\PP^4}(2))=6$ and hence $C=\B$ and    
\begin{eqnarray}
\B&=& V(x_2^2-x_1x_3,\ x_2x_3-x_1x_4,\ x_0x_4-x_1x_3,\ x_3^2-x_2x_4,\ x_0x_2-x_1^2,\ x_0x_3-x_1x_2), \\
\Q&=& V(y_0y_2-y_1y_5+y_3y_4).
\end{eqnarray}
\end{case}
\begin{case}[$\deg(C)=4$,   $C$ degenerate] This case is impossible 
by 
 Lemma \ref{prop: degenerate component} and Proposition \ref{prop: type 2-1}.
\end{case}
\begin{case}[$\deg(C)=3$,   $\langle C \rangle=\PP^3$] 
Modulo a change of coordinates, 
$C$ is the twisted cubic curve parameterized by
$[s,t]\in\PP^1\mapsto [s^3,s^2t,st^2,t^3,0]\in V(x_4)$  
and, 
by the reduction obtained, since
$ h^0(\PP^4,\I_{C}(2))=h^0(\PP^4,\I_{\B}(2))+2$,
it follows
$\B=C\cup r$, 
where $r$ is a line that intersects $C$ in a single point transversely.
We can choose the intersection point to be $p=[1,0,0,0,0]$
and then we have
$r=\{ [s+q_0t,q_1t,q_2t,q_3t,q_4t]: [s,t]\in\PP^1\}$,
for a some  $q=[q_0,q_1,q_2,q_3,q_4]\in\PP^4$.     
By Proposition \ref{prop: type 2-1} we have $q_4\neq0$ and 
only for simplicity of notation we take
$q=[0,0,0,0,1]$,   hence $r=V(x_1,x_2,x_3)$.  So we obtain 
\begin{eqnarray}
\B&=&V( x_1^2-x_0x_2,\  x_3x_4,\  x_0x_3-x_1x_2,\  x_2x_4,\   x_1x_3-x_2^2,\  x_1x_4),\\
\Q&=& V(-y_4y_5+y_2y_3+y_0y_1).
\end{eqnarray}
\end{case}
\begin{case}[$\deg(C)=3$,   $\langle C \rangle=\PP^2$] It is impossible 
because otherwise 
$\B$ would contain
the entire plane spanned by $C$.
\end{case}
\begin{case}[$\deg(S)=1$] Since 
planes, conics and lines impose to the quadrics 
 respectively  $6$, $5$ and $3$ conditions, 
we have two subcases:
\begin{subcase}[$\B=S\cup\Gamma$,   $\Gamma$ conic, $\#(\Gamma\cap S)=2$] This
 is impossible
because otherwise there exists a line cutting $\B$ in exactly $3$ points.  
\end{subcase}
\begin{subcase}[$\B=S\sqcup l$,   $l$ line] This contradicts 
the birationality of $\varphi$. 
\end{subcase}
\end{case}
\begin{case}[$\deg(C)=2$,   $C$ irreducible] Let $C', l, l'$ be 
respectively 
eventual conic and eventual lines contained in $\B$.
We discuss the subcases:  
 \begin{subcase}[$\B=C\cup C'$,    $\#(C\cap C')=1$] In this case, 
we would have that  $\rk(\Q)=4$, against the  hypothesis.
 \end{subcase}
\begin{subcase}[$\B=C\cup C'\cup l$,    $\#(C\cap C')=2$,   $\#(C\cap l)=\#(C'\cap l)=1$] 
 $\B \supseteq C\cup C'\cup l$   implies that $\B$ is a quadric of codimension $2$.
\end{subcase}
\begin{subcase}[$\B=C\cup l\cup l'$,   $\#(l\cap l')=0$,   $\#(C\cap l)=\#(C\cap l')=1$] 
This case is really possible and, 
modulo a change of coordinates, 
we have
\begin{eqnarray}
 \B&=&V(x_2^2+x_0x_1,\ x_3x_4,\ x_1x_3,\ x_2x_4,\ x_2x_3,\ x_0x_4),\\
 \Q&=&V(-y_2y_5-y_3y_4+y_0y_1).
\end{eqnarray}
\end{subcase}
\end{case}
\begin{case}[$\deg(C)=2$,   $C$ reducible] By
  the reduction obtained,
$\B$ contains  two other skew lines, 
each of which intersects $C$ at a single point. 
This case is really possible and, 
modulo a change of coordinates, 
we have
\begin{eqnarray}
 \B&=&V(x_1x_2,\ x_3x_4,\ x_0x_3,\ x_2x_4,\ x_2x_3,\ x_0x_4-x_1x_4),\\
\Q&=&V(-y_4y_5+y_2y_3-y_0y_1).
\end{eqnarray}
\end{case}
\end{proof}

\begin{figure}
\begin{center}
\includegraphics[width=0.75\textwidth]{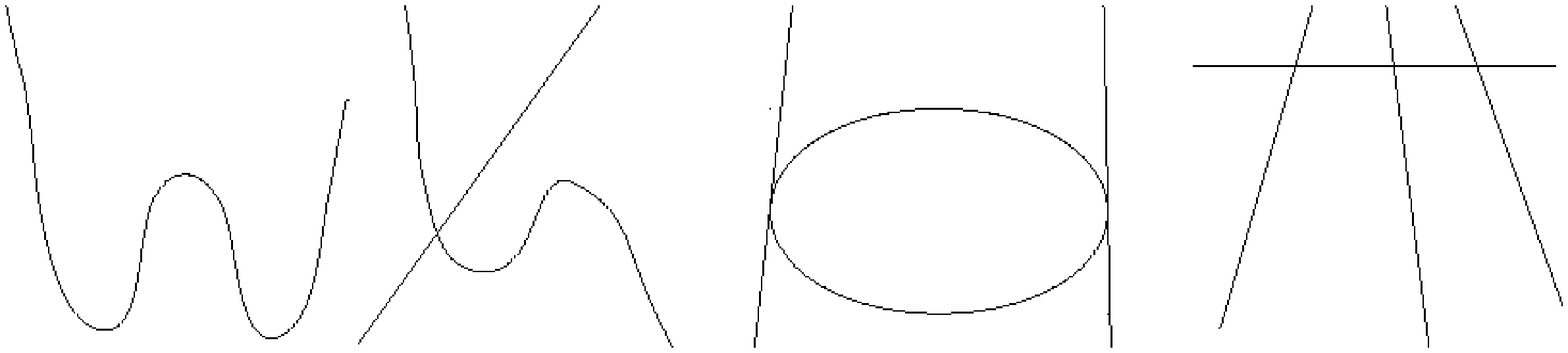}
\end{center}
\caption{Base loci when $n=4$.}
\label{fig: fig_base_loci_n=4}   
\end{figure}

Example \ref{example: extra} suggests that a possible generalization
of  Proposition \ref{prop: P4}
 to the case where $\B$ in nonreduced and $rk(\Q)< 5$ 
may not be trivial.
\begin{example}\label{example: extra}
The rational map $\varphi:\PP^4\dashrightarrow\PP^5$ defined by 
$$\varphi([x_0,x_1,x_2,x_3,x_4])=
[x_0^2, -x_0x_1, -x_0x_2, x_1^2-x_0x_3, 2x_1x_2-x_0x_4, x_2^2] , $$
is birational 
into  its image, which is the quadric of rank $3$,
$\Q=V(y_0y_5-y_2^2)$.
If  $\pi:\PP^5\dashrightarrow \PP^4$ is the projection 
from the point  $[0,0,0,0,0,1]$, then the composition
 $\pi\circ\varphi:\PP^4\dashrightarrow \PP^4$ is an involution.
The base locus $\B$ of $\varphi$ is everywhere nonreduced, 
$(\B)_{\mathrm{red}}=V(x_0,x_1,x_2)$ and 
$P_{\B}(t)=4t+1$.
\end{example}

\begin{acknowledgements}
The author wishes to thank 
Prof. Francesco Russo for his 
indispensable suggestions about the topics in this paper.
\end{acknowledgements}

\bibliographystyle{spmpsci} 
\bibliography{bibliography.bib}
\addcontentsline{toc}{section}{References}

\end{document}